\documentclass[12pt]{amsart}
\usepackage{amsmath}
\usepackage{amsfonts}
\usepackage{amsthm}
\usepackage{enumerate}
\usepackage{graphicx}
\usepackage{overpic}
\usepackage{a4wide}
\usepackage{amssymb} 
\usepackage{pictex}
\usepackage{rotating}  
\usepackage{xcolor}
\usepackage{etex}
\usepackage{cite} %pour liste de ref. remplac. par un tiret
\usepackage{mathdots}
\usepackage{comment}
\usepackage{tikz}
\usetikzlibrary{shapes,arrows,automata}
\usetikzlibrary{calc,backgrounds}

\numberwithin{equation}{section}

\newtheorem{theorem}{Theorem}[section]

\newtheorem{proposition}[theorem]{Proposition}
\newtheorem{lemma}[theorem]{Lemma}
\newtheorem{corollary}[theorem]{Corollary}
\newtheorem{Definition}[theorem]{Definition}
\newenvironment{definition}{\begin{Definition}\rm}{\end{Definition}}
\newtheorem{Remark}[theorem]{Remark}
\newenvironment{remark}{\begin{Remark}\rm}{\end{Remark}}
\newtheorem{RHproblem}[theorem]{RH problem}

\newtheorem{Example}[theorem]{Example}

\newcommand{\C}{\mathbb{C}}
\newcommand{\D}{\mathbb D}

\newcommand{\F}{\mathbb F}

\newcommand{\N}{\mathbb{N}}
\newcommand{\R}{\mathbb{R}}
\newcommand{\Q}{\mathbb{Q}}

\renewcommand{\P}{\mathbb{P}}

\newcommand{\DD}{\mathcal D}

\newcommand{\NN}{\mathcal N}
\newcommand{\PP}{\mathcal P}
\newcommand{\MS}{\mathcal S}

\newcommand{\UU}{\mathcal U}
\newcommand{\ZZ}{\mathcal Z}

\newcommand{\OO}{\mathcal O}

\def \deg{\mbox{{\rm deg} }}

\renewcommand{\bar}{\overline}
\renewcommand{\tilde}{\widetilde}
\renewcommand{\hat}{\widehat}

\usepackage[bookmarksopen, naturalnames]{hyperref}

\makeatother

\begin{document}
\title{Generic properties of Pad\'e approximants and Pad\'e universal series}
\author{S. Charpentier, V. Nestoridis and F. Wielonsky}
\address{St\'ephane Charpentier, 
Institut de Math\'ematiques de Marseille, UMR 7373, Aix-Marseille Universit\'e, Technop\^ole Ch\^ateau-Gombert, 39 rue F. Joliot Curie, 13453 Marseille Cedex 13, FRANCE}
\email{stephane.charpentier.1@univ-amu.fr}
\address{Vassili Nestoridis, 
Department of Mathematics, University of Athens Panepisitmioupolis, 15784 Athens, GREECE}
\email{vnestor@math.uoa.gr}
\address{Franck Wielonsky, 
Institut de Math\'ematiques de Marseille, UMR 7373, Aix-Marseille Universit\'e, Technop\^ole Ch\^ateau-Gombert, 39 rue F. Joliot Curie, 13453 Marseille Cedex 13, FRANCE}
\email{franck.wielonsky@univ-amu.fr}
\maketitle
\begin{abstract}
We establish properties concerning the distribution of poles of Pad\'e approximants, which are generic in Baire category sense. We also investigate
%then complement previous studies of 
Pad\'e universal series, an analog of classical universal series, where Taylor partial sums are replaced with Pad\'e approximants. In particular, we complement previous studies on this subject by exhibiting dense or closed infinite dimensional linear subspaces of analytic functions in a simply connected domain of the complex plane, containing the origin, whose all non zero elements are made of Pad\'e universal series. We also show how Pad\'e universal series can be built from classical universal series with large Ostrowski-gaps.

\end{abstract}
\section{Introduction}
Many questions related to polynomial or rational approximations have been addressed from the point of view of Baire category, see e.g. \cite{Bar,Her,Her2,BOR,BOR2}. For instance, P.B. Borwein established in \cite{BOR} that the set 
%answered a question by Gonchar \cite{Gon1} proving not only that there exists an 
of entire functions $f$, for which the poles of their best rational approximations on $[-1,1]$ are dense in $\C$, is a residual set.
%, but also that this property is generic. 
%Here by generic we mean that the set of such functions is dense in the set $H(\C)$ of all entire functions, endowed with the topology of uniform convergence on compact sets.
He also proved that Baker's conjecture, on the existence of a subsequence in the $m$-th row of the Pad\'e table that converges locally uniformly, holds true generically in the set of entire functions.
%Similarly Baker conjectured that if $f$ is an entire function then it can be approximated by some subsequence of the $m$th row of its Pad\'e table uniformly on compact sets. The case $m=0$ is that of Taylor series, $m=2$ was proved in \cite{BGM}. Borwein proved the following result.
%
%\begin{Theorem}[\cite{Bor}]\label{thm1-intro}Let $\left(p_n,q_n\right)_{n\in \N}$ be a sequence of pairs of positive integers such that $\left(p_n-q_n\right)_n$ tends to infinity. The set of entire functions $f$ for which there exists an infinite subset $I$ of positive integers so that the $\left(p_n,q_n\right)_{n\in I}$ Pad\'e approximants converge to $f$ in $H(\C)$ is residual in $H(\C)$.
%\end{Theorem} 
Fournodavlos and Nestoridis recently improved this result 
%by showing that the conclusion still holds if one only assumes the sequence $\left(p_n\right)_{n\in \N}$ to be unbounded 
in \cite{FNes}. In the first part of the present paper, we derive similar results about the density of poles for the case of Pad\'e approximants. We also show that, generically, the poles of a subsequence of the Pad\'e approximants go to infinity.

%Questions raised about best rational approximations can be asked in the context of Pad\'e approximants. In particular it is natural to wonder whether there exist holomorphic functions on some domain $\Omega$ containing $0$ whose set of poles of Pad\'e approximants -if they exist- is dense in $\C$ or in some subset of $\C$. Thanks to Baire Category Theorem we give a positive anwser to this question. Combining this to Theorem \ref{DFN-main-intro}, we immediately deduce the existence of a residual subset of $H(\Omega)$ consisting of Pad\'e universal series whose set of poles of its Pad\'e approximants  is dense in $\C$.
The remaining, and main part, of the paper is devoted to the study of Pad\'e universal series. We first need to introduce a few notations and the notion of (classical) universal series.
For $\Omega$, a simply connected domain in $\C$, containing the origin, we denote by $H(\Omega)$ the set of holomorphic functions on $\Omega$ endowed with the topology of locally uniform convergence. For a compact subset $K$ of $\C$, 
%of connected complement, 
we also denote by $A(K)$ the set of continuous functions in $K$, holomorphic in its interior, endowed with the topology of uniform convergence. For a power series 
$f=\sum_{n\geq 0}a_{n}z^{n}$, we denote by $S_{N}(f)$ the $N$-th partial sum 
$\sum_{n= 0}^{N}a_{n}z^{n}$ of $f$.
We recall that a {\it universal series} in $\Omega$, with respect to an increasing sequence $\mu=(p_{n})_{n}$ of positive integers, is a function $f$ in $H(\Omega)$ such that, for any compact set $K\subset\C \setminus \Omega$, with connected complement, and any function $h$ in $A(K)$, there is a subsequence $(\lambda_{n})_{n}$ of $\mu$ such that the partial sums $S_{\lambda_{n}}(f)$ converges to $h$ in $A(K)$, while they converge to $f$ in $H(\Omega)$, cf. \cite{bgnp,Luh2,Nes}.
%While quasi-all power series of radius of convergence $R$ can be approximated by some subsequence of its Pad\'e table uniformly on each compact subset of the disc $D(0,R)$, is there any such power series whose Pad\'e approximants \emph{diverge as much as possible} outside $D(0,R)$? This question was still addressed in a more general setting using the point of view of category which is quite popular in universal series context (see  \cite{bgnp,Nes}).
%Let us give the definitions of what we call universal series and Pad\'e universal series. For $\Omega$ a domain of the complex plain, we denote by $H(\Omega)$ the set of holomorphic functions on $\Omega$ endowed with the topology of uniform convergence on compact subsets.
%
%\begin{definition}[Universal series]\label{defi-univ-Taylor-series-intro} Let $\Omega \subset \C$ be a simply connected domain containing $0$. A universal series is a function $f$ in $H(\Omega)$ such that, for every compact set $K\subset\C \setminus \D$ with connected complement and every continuous function $h$ on $K$ holomorphic in its interior, there exists an infinite subset $I$ of $\N$ with the following properties:
%\begin{enumerate}[a)]
%\item The partial sums $S_{n}\left(f\right)=\sum _{k=0}^{n}a_kz^k$ of the Taylor expansion $\sum _{k=0}^{\infty}a_kz^k$ of $f$ at $0$ converges to $h$ uniformly on $K$ as $n$ tends to $\infty$, $n\in I$;
%\item $S_{n}\left(f\right)$ converges to $f$ uniformly on every compact subset of $\Omega$ as $n$ tends to $\infty$, $n\in I$.
%\end{enumerate}
%We denote by $U(\Omega)$ the set of such universal series.
%\end{definition}
It was proved in \cite{Nes2} that the set $U^{\mu}(\Omega)$ of universal series is residual in $H(\Omega)$, see also \cite{Nes}. 
%
%We recall that a subset $A$ of a complete metric space $X$ is residual in $X$ if $A$ contains a dense countable intersection of open sets.
%
%\begin{theorem}[\cite{Nes2}]\label{thm-univ-Taylor-series-intro}Let $\Omega\subset \C$ be a simply connected domain containing $0$. $U(\Omega)$ is residual in $H(\Omega)$.
%\end{theorem}

Several articles recently dealt with the notion of \emph{Pad\'e universal series}, an analog of universal series, where the role of partial Taylor sums is played by Pad\'e approximants, see e.g. \cite{DFN,FNes,Nes3}. For a power series $f$ at the origin, we denote by $[f;p_{n}/q_{n}]$ the Pad\'e approximant of degree $(p_{n},q_{n})$ to $f$, see Section \ref{Prelim-Pade} for more details on Pad\'e approximants.
\\[\baselineskip]
{\bf Definition of Pad\'e universal series }%[Pad\'e universal series]\label{defin-Pade-intro}
%Let $\Omega\subset \C$ be a simply connected domain containing $0$. 
Let $\MS :=\left(p_n,q_n\right)_n$ be a sequence of pairs of non negative integers. A Pad\'e universal series in $\Omega$, with respect to $\MS$, is a function $f\in H(\Omega)$ such that, for every compact set $K$ in $\C \setminus \Omega$, with connected complement, and every function $h$ in $A(K)$, there exists a subsequence $\MS'$ of $\MS$ such that:
\begin{enumerate}[i)]
%\item The $\left(p_{n},q_{n}\right)$ Pad\'e approximant $\left[f;p_{n}/q_{n}\right]$ of $f$ at $0$ exists for every $n\in I$;
\item $[f;p_{n}/q_{n}]$ converges to $h$ in $A(K)$ along the subsequence $\MS'$;%as $n$ tends to $\infty$, $n\in I$;
\item $[f;p_{n}/q_{n}]$ converges to $f$ in $H(\Omega)$ along the subsequence $\MS'$.
% uniformly on every compact subset of $\Omega$ as $n$ tends to $\infty$, $n\in I$;
\end{enumerate}
Assuming that the sequence $(p_{n})_{n}$ is unbounded, the set
$\UU ^{\MS}(\Omega)$ of Pad\'e universal series is residual in $H(\Omega)$, see
\cite[Theorem 3.1]{DFN}. Actually, it is easy to see, by a Rouch\'e type argument in the complement of $\Omega$, that the hypothesis on the sequence $(p_{n})_{n}$ is mandatory for that result to hold true.
A stronger notion of Pad\'e universal series has been considered, where the Pad\'e approximants approach any rational functions on any compact sets (not necessarily with connected complement) of $\C\setminus\Omega$ with respect to the chordal metrics \cite{Nes3}. We shall not pursue this direction here.

%A result similar to Theorem \ref{DFN-main-intro} can be proved using \cite{Nes3} and a slight modification of the argument given above to show that c) implies a) in the previous theorem (this time both sequences $\left(p_n\right)_n$ and $\left(q_n\right)_n$ have to be unbounded). In this paper we won't investigate such kind of Pad\'e universal series.

In the investigation of universality, it is natural to seek some significant subsets of the set of universal elements having a linear structure. For instance, one may ask if there exists a linear subspace, made of universal elements (except for the zero one), which is dense in the ambient space. If true, one says that the universal elements are {\it algebraically generic}. One may also ask about the existence of a closed and infinite dimensional subspace, in which case one speaks of {\it spaceability} of the universal elements. Since the last decade, there has been a growing interest in this line of research. We refer the reader to \cite{Aron1,Aron2,Bay0,BO,BPS,G-E,Montes,Men2} and the references therein. The non linearity of Pad\'e approximation makes the above questions a little bit trickier for Pad\'e universal series. 
%In universality and especially for classical universal series, this is usual to ask if there exist dense (resp. closed infinite dimensional) subspaces whose nonzero elements share all the same universal properties. If such a phenomenom occurs we say that the set of such universal elements is algebraically generic (resp. spaceable). Concerning Pad\'e approximants, it is likely that given two power series such that there exist $\left(p,q\right)$ Pad\'e approximants, we can find a nonzero linear combination of them whose $\left(p,q\right)$ Pad\'e approximant does not exist. Roughly speaking the \emph{map} sending a power series to its $\left(p,q\right)$ Pad\'e approximants is not linear (it is in fact not well defined). 
Nevertheless, by considering particular series whose Pad\'e approximants, along a row of the Pad\'e table, have prescribed denominators, we answer both questions in the affirmative, see Theorems \ref{alg-gene} and \ref{spaceability}.

Our method of proof is actually an adaptation of the one used for classical universal series. This makes it possible to derive results on Pad\'e universal series from related results about universal series.
% it appear a useful connexion between universal Taylor series and Pad\'e universal series in the sense of Definition \ref{defin-Pade-intro}. Indeed we show that main known results on Pad\'e universality on simply connected domains can be derived from universal series. 
In particular, from the recent work \cite{CM}, we know that there exists a 
%dense $G_{\delta}$ 
residual subset of $H(\D)$ (or more generally of $H(\Omega)$), made of universal Taylor series with so-called {\it large} Ostrowski-gaps. Upon performing a simple division, we are then lead to
Pad\'e universal series. Together with Baire category theorem, this allows us to recover the fact that $\UU^{\MS}(\D)$ is residual in $H(\D)$, \cite[Theorem 3.1]{DFN}.
%Theorem \ref{DFN-main-intro} only from universal series.
%In particular this means that the genericity of Pad\'e universal series on simply connected domains can be obtained only from polynomial approximation, and not necessarily rational approximation any more.

\medskip{}

The paper is organized as follows. In Section \ref{Prelim-Pade}, we recall the main properties of Pad\'e approximants. Section \ref{beh-poles} deals with generic properties of poles of Pad\'e approximants. In Section \ref{part-set}, we study Pad\'e universal series whose Pad\'e approximants have prescribed poles, and use these series to exhibit algebraically generic sets, as well as spaceable sets, of Pad\'e universal series. We also prove that Pad\'e universal series whose Pad\'e approximants have asymptotically prescribed poles are generic.
%The technics used allows us to provide a constructive way to produce Pad\'e universal series whose approximating Pad\'e approximants have dense set of poles. 
In the final section \ref{TtoP}, we show how to build Pad\'e universal series from classical universal series with large Ostrowski gaps, and use this method to recover the genericity of Pad\'e universal series.

\medskip{}

\noindent{}\emph{Notations.} 
%In the sequel we identify a formal power series $f=\sum _{n=0}^{\infty}a_nz^n$ with the sequence $\left(a_n\right)_n$ in $\C^{\N}$. We use the notation $S_N(f)$, $N\in \N$, to denote the $n$-th partial sum $\sum _{n=0}^N a_n z^n$. 
Throughout, we will denote by $[f;p/q]$ the Pad\'e approximant of degree $(p,q)$ of a power series $f$ at the origin. We will also denote by $P_{p,q}(f)$ and $Q_{p,q}(f)$ the numerator and the denominator of $[f;p/q]$, written in irreducible form.
%respectively (see Section \ref{section-poles-denses}).

Let $P=\sum _{k=0}^pa_kz^k$ be a polynomial. The valuation of $P$ is the smallest index $k$ such that $a_k\neq 0$. We denote it $\text{val}(P)$. 
As usual the degree of $P$, denoted by deg$(P)$, is the largest $k$ with $a_k\neq 0$.

%Given a compact set $K$ of $\C$ we denote by $A(K)$ the space consisting of those continuous functions on $K$ which are holomorphic in the interior of $K$, endowed with the topology of uniform convergence given by the supremum norm on $K$ denoted by $\left\Vert \cdot \right\Vert _K$.
Finally, we recall that a subset of a topological space is a $G_{\delta}$ set if it can be expressed as an intersection of countably many open sets. It is residual if it is the intersection of countably many sets with dense interiors. A meagre subset is the complement of a residual set. A property is generic if it holds on a residual set.
%\section{Pad\'e approximants with dense poles}\label{section-poles-denses}
\section{Preliminaries on Pad\'e Approximants}\label{Prelim-Pade}
In this section, we recall the definition of Pad\'e approximants and give some of their properties. Basic references for Pad\'e approximants are e.g. \cite{BGM, GRA, NS, Per}.
\begin{Definition}
Let $S(z)=\sum_{k\geq 0} a_{k}z^{k}$ be a formal power series with complex coefficients.
% with $a_{0}\neq 0$. 
A Pad\'e approximant of type $(m,n)$ of $S$ is a rational function $P_{m}/Q_{n}$, with
$\deg P_{m}\leq m$, $\deg Q_{n}\leq n$, $(P_{m},Q_{n})$ coprime, $Q_{n}(0)=1$
and such that
\begin{equation}\label{Pade}
S(z)-P_{m}/Q_{n}(z)=\OO(z^{m+n+1}).
\end{equation}
The (algebraic) $\OO$ symbol indicates that the expression on the left side is a power series beginning with a power $z^{l}$, $l\geq m+n+1$.
\end{Definition}

Here, it is understood that two polynomials are coprime if they have no common roots. In particular, the pair $(0,1)$ is coprime while the pair $(0,z)$ is not coprime.
Also, we agree to represent the zero function in a unique way as a rational function, namely the rational function $0/1$ where $0$ and $1$ are the constant polynomials of degree 0. 

Note that, if a Pad\'e approximant exists, it can only be unique. Indeed, if we have two solutions $P_{m}/Q_{n}$ and $\tilde{P}_{m}/\tilde{Q}_{n}$ of (\ref{Pade}), then 
$$(\tilde{Q}_{n}P_{m}-Q_{n}\tilde{P}_{m})(z)=\OO(z^{m+n+1}),$$
so that $\tilde{Q}_{n}P_{m}-Q_{n}\tilde{P}_{m}=0$ and the two rational functions are equal. Note also that the assumption $Q_{n}(0)=1$ is not a real restriction since if $Q_{n}(0)=0$ then $P_{m}(z)$ would have to cancel the zeros of $Q_{n}(z)$ at the origin in order that (\ref{Pade}) has a meaning.

The linearized version of (\ref{Pade}) is
\begin{equation}\label{pade}
Q_{n}(z)S(z)-P_{m}(z)=\OO(z^{m+n+1}).
\end{equation}
Because of the assumption $Q_{n}(0)=1$, it is actually equivalent to (\ref{Pade}). The number $n+m+1$ of free coefficients on the left of (\ref{pade}) equals the number of vanishing coefficients on the right. The coefficients of $Q_{n}$ are used to make the coefficients of $z^{m+1},\ldots,z^{m+n}$ in the expansion of $Q_{n}(z)S(z)-P_{m}(z)$ vanish. Then the coefficients of $P_{m}$ are chosen so that the coefficients of $1,\ldots,z^{m}$ vanish. Setting
$Q_{n}(z)=q_{n}z^{n}+\cdots+1$, one gets the linear system of $n$ equations with $n$ unknowns,  
\begin{equation}\label{syst}
\begin{pmatrix}a_{m-n+1} & \ldots & a_{m}\\
\vdots & \vdots & \vdots\\
a_{m} & \ldots & a_{m+n-1}\end{pmatrix}
\begin{pmatrix}q_{n} \\ \vdots \\ q_{1}\end{pmatrix}=-\begin{pmatrix}a_{m+1} \\ \vdots \\ a_{m+n} \end{pmatrix}.
\end{equation}
We denote by $C_{m,n}$ the (Hankel) determinant of the above system,
\begin{equation}\label{Hankel}
C_{m,n}:=\begin{vmatrix}a_{m-n+1} & \ldots & a_{m}\\
\vdots & \vdots & \vdots\\
a_{m} & \ldots & a_{m+n-1}\end{vmatrix}.
\end{equation}
The Pad\'e approximant $P_{m}/Q_{n}$ of $S$, if it exists, admits an explicit expression, originally given by Jacobi, in terms of determinants, namely $P_{m}/Q_{n}$ equals the quotient of the following two polynomials, written in determinantal form as
\begin{equation}\label{deter}
\hat{P}_{m}(z)=\begin{vmatrix}a_{m-n+1} & a_{m-n+2}& \ldots & a_{m+1}\\
\vdots & \vdots & \vdots & \vdots\\
a_{m} & a_{m+1} & \ldots & a_{m+n}\\
\sum\limits_{i=0}^{m-n}a_{i}z^{n+i} & \sum\limits_{i=0}^{m-n+1}a_{i}z^{n+i-1} & \ldots & 
\sum\limits_{i=0}^{m}a_{i}z^{i}
\end{vmatrix},%\quad
\hat{Q}_{n}(z)=\begin{vmatrix}a_{m-n+1} & \ldots & a_{m+1}\\
\vdots & \vdots & \vdots\\
a_{m} & \ldots & a_{m+n}\\
z^{n}& \ldots & 1\end{vmatrix}.
\end{equation}
Indeed, it is easy to verify that the power expansion of $\hat{Q}_{n}S(z)$ has no powers of $z^{l}$, $m+1\leq l\leq m+n$, and that the one of $(\hat{Q}_{n}S-\hat{P}_{m})(z)$ has no powers of $z^{l}$, $0\leq l\leq m$. Hence the linearized version (\ref{pade}) is always satisfied by $\hat{P}_{m}$ and $\hat{Q}_{n}$ and, if the Pad\'e approximant exists, we thus have $P_{m}/Q_{n}=\hat{P}_{m}/\hat{Q}_{n}$.
\begin{proposition}\label{equiv-pade}
Let $S(z)=\sum_{k\geq 0} a_{k}z^{k}$ be a formal power series. The following three assertions are equivalent:\\
(i) The determinant $C_{m,n}$ is nonzero.\\
(ii) The Pad\'e approximant $P_{m}/Q_{n}$ (which, by definition, is irreducible) of type $(m,n)$ of $S$ exists and $\deg P_{m}=m$ or $\deg Q_{n}=n$.\\
(iii) The polynomials $\hat{P}_{m}$ and $\hat{Q}_{n}$ are coprime and $\deg \hat{P}_{m}=m$ or $\deg \hat{Q}_{n}=n$.\\
(iv) The polynomials $\hat{P}_{m}$ and $\hat{Q}_{n}$ are coprime.\\ 
If the above assertions hold true then $\hat{P}_{m}(z)=\hat{Q}_{n}(0)P_{m}(z)$ and $\hat{Q}_{n}(z)=\hat{Q}_{n}(0)Q_{n}(z)$. Moreover, if $C_{m,n}=0$ and the Pad\'e approximant $P_{m}/Q_{n}$ of type $(m,n)$ of $S$ exists, then $\hat{P}_{m}(z)=T(z)P_{m}(z)$ and $\hat{Q}_{n}(z)=T(z)Q_{n}(z)$ where $T$ is a %nonzero 
polynomial such that $T(0)=0$. 
\end{proposition}
\begin{proof}
We start with the equivalence of (i) and (ii).
if $C_{m,n}\neq 0$ then the system (\ref{syst}) has a unique solution and the Pad\'e approximant exists. If $\deg P_{m}<m$ and $\deg Q_{n}<n$ then e.g. $(z+1)P_{m}(z)$ and $(z+1)Q_{n}(z)$ would also be a solution to (\ref{pade}) contradicting the uniqueness of the solution to the system (\ref{syst}). Conversely, assume that the Pad\'e approximant exists and $\deg P_{m}=m$ or $\deg Q_{n}=n$. If $C_{m,n}=0$, then the system (\ref{syst}) can only have infinitely many solutions. Two rational functions
$P_{m}^{(1)}/Q_{n}^{(1)}$ and $P_{m}^{(2)}/Q_{n}^{(2)}$, obtained from two distinct solutions $Q_{n}^{(1)}\neq Q_{n}^{(2)}$
of (\ref{syst}) represent the same irreducible rational function $P_{m}/Q_{n}$. 
From the assumption that $\deg P_{m}=m$ or $\deg Q_{n}=n$, the polynomials $Q_{n}^{(1)}$ and $Q_{n}^{(2)}$ can only differ from $Q_{n}$ by constant factors, and from
the normalization $Q_{n}^{(1)}(0)=Q_{n}^{(2)}(0)=1$, these constant factors have to be equal to 1. Hence, $Q_{n}^{(1)}=Q_{n}^{(2)}$, which gives a contradiction.
%the polynomials $Q_{n}^{(1)}$ and $Q_{n}^{(2)}$ have to differ by a nonconstant polynomial factor, and the same holds true for $P_{m}^{(1)}$ and $P_{m}^{(2)}$. Hence, $P_{m}/Q_{n}$ is a rational function with
%$\deg P_{m}<m$ and $\deg Q_{n}<n$, contradicting assertion (ii). 
We next show the equivalence of (i) or (ii) with (iii).
 If $C_{m,n}\neq 0$, one has $\hat{Q}_{n}(0)\neq 0$ and $P_{m}/Q_{n}=\hat{P}_{m}/\hat{Q}_{n}$ because $\hat{P}_{m}$ and $\hat{Q}_{n}$ satisfy (\ref{pade}), so that (ii) implies (iii) and $Q_{n}(z)=\hat{Q}_{n}(z)/\hat{Q}_{n}(0)$, $P_{m}(z)=\hat{P}_{m}(z)/\hat{Q}_{n}(0)$. Conversely, if (iii) holds true, then $\hat{Q}_{n}(0)\neq 0$ since otherwise $z$ would be a common factor of $\hat{P}_{m}$ and $\hat{Q}_{n}$. Hence $P_{m}/Q_{n}$ exists and equals $\hat{P}_{m}/\hat{Q}_{n}$. Since both fractions are irreducible, $\deg P_{m}=\deg\hat{P}_{m}$ and $\deg Q_{n}=\deg\hat{Q}_{n}$ and assertion (ii) is satisfied. 
Clearly (iii) implies (iv) and (iv) implies (i) because if $C_{m,n}=0$, we see from the determinantal expressions that $\hat{P}_{m}(0)=\hat{Q}_{n}(0)=0$ which would be a contradiction.

Finally, if $C_{m,n}=0$ and the Pad\'e approximant exists, we have $P_{m}/Q_{n}=\hat{P}_{m}/\hat{Q}_{n}$ because $\hat{P}_{m}$ and $\hat{Q}_{n}$ satisfy (\ref{pade}). Since
%It is also clear from the determinantal expressions that 
$\hat{P}_{m}(0)=\hat{Q}_{n}(0)=0$, we obtain the last assertion of the proposition. 
\end{proof}
Throughout, we denote by $\F$ be the space of formal power series $\sum_{k\geq 0}a_{k}z^{k}$ around the origin. 
\begin{definition}
We define $\DD_{m,n}$ as the subset of power series in $\F$ (or in $H(\Omega)$, according to the context) such that the associated Hankel determinant $C_{m,n}$ does not vanish, and $\NN_{m,n}$ as the subset of $\DD_{m,n}$ of series $S$ such that $\deg Q_{m,n}(S)=n$ (one then says that the degree $(m,n)$ is \emph{normal} for $S$). When considering a sequence of degrees $(m_{k},n_{k})$ indexed by an integer $k\geq 0$, we shall usually simplify the notations to $\DD_{k}$ and $\NN_{k}$. 
\end{definition}
%For a series $S$ around the origin, we shall write $[S;m/n]$ for the Pad\'e approximant of $S$ of type $(m,n)$.

We identify a formal power series $\sum_{k\geq 0}a_{k}z^{k}$ around the origin, with the sequence $(a_{k})_{k}$ in $\C^{\N}$, and we endow $\F$ with the topology of the cartesian product. This topology can be equivalently defined by the family of semi-norms $p_{j}((a_{k})_{k})=|a_{j}|$, $j\geq 0$, and the associated distance
$$d((a_{k})_{k},(b_{k})_{k})=\sum_{k=0}^{\infty}\frac{1}{2^{k}}\frac{|a_{k}-b_{k}|}{1+|a_{k}-b_{k}|},$$
which makes $(\F,d)$ into a complete metric space. 

Let $\Omega$ be a simply connected domain, containing the origin,
%$0\in\Omega\subset\C$, let $H(\Omega)$ be the space of analytic functions in $\Omega$, endowed with the topology of uniform convergence on compact subsets of $\Omega$. 
and let $(K_{k})_{k\geq 0}$ be an exhausting sequence of compact subsets of $\Omega$. We endow the space $H(\Omega)$ with the topology of uniform convergence on the compact sets $K_{k}$, $k\geq 0$. This is a complete metric space, associated with the distance
$$d_{\infty}(f,g)=\sum_{k=0}^{\infty}\frac{1}{2^{k}}\frac{p_{k}(f-g)}{1+p_{k}(f-g)},$$
where $p_{k}$ denotes the semi-norm $p_{k}(f)=\sup_{z\in K_{k}}|f(z)|$. The topology of $H(\Omega)$ is stronger than that of the ambient space $(\F,d)$. 
\begin{proposition}\label{D-N}
The sets $\DD_{m,n}$ and $\NN_{m,n}$ are open and dense subsets of $\F$ and $H(\Omega)$, endowed with their respective topology.
\end{proposition}
\begin{proof}
The set $\DD_{m,n}$ (resp. $\NN_{m,n}$) is characterized by the fact that the Hankel determinant $C_{m,n}\neq 0$ (resp. $C_{m,n}\neq0$ and $C_{m+1,n}\neq0$).
These Hankel determinants are continuous functions of a finite number of coefficients of $f$ so $\DD_{m,n}$ and $\NN_{m,n}$ are open subsets of $\F$ and $H(\Omega)$. Actually, they are polynomial functions of the coefficients, so, by an analyticity argument, $\DD_{m,n}$ and $\NN_{m,n}$ are also dense in $\F$ and $H(\Omega)$.
The sets $\DD_{m,n}$ and $\NN_{m,n}$, as subsets of $H(\Omega)$, are open in $H(\Omega)$ since all functions in a given neighborhood of a $f\in H(\Omega)$ have, by Cauchy's estimates, Taylor coefficients close to those of $f$. The sets $\DD_{m,n}$ and $\NN_{m,n}$ are also dense in $H(\Omega)$ because if 
$f\in H(\Omega)$ does not belong to $\DD_{m,n}$ or $\NN_{m,n}$, any neighborhood of $f$ contains a set of the form
$$\{f(z)+\sum_{i=m_{k}-n_{k}+1,~i\geq 0}^{m_{k}+n_{k}-1}r_{i}z^{i},~|r_{i}|\leq\epsilon\},$$
for $\epsilon$ small enough, which itself contains an element in $\DD_{m,n}$ or $\NN_{m,n}$.
\end{proof}
In the next section, the following proposition, whose proof is immediate from the definition of Pad\'e approximants and Proposition \ref{equiv-pade}, will be useful.
\begin{proposition}\label{inv-Pade}
Let $S(z)=\sum_{k\geq 0}a_{k}z^{k}$ be a power series with $a_{0}\neq 0$. If the Pad\'e approximant of type $(m,n)$ for $S$ exists, then the Pad\'e approximant of type $(n,m)$ for the reciprocal $1/S$ of $S$ also exists and 
$$[1/S;n/m]=1/[S;m/n].$$ 
Moreover, $S\in\DD_{m,n}$ if and only if $1/S\in\DD_{n,m}$.
\end{proposition}
\section{Behavior of poles of Pad\'e approximants}\label{beh-poles}
Our first result states that, generically, the poles (or the zeros) of the Pad\'e approximants of a given series form a dense subset of $\C$. One may interpret this property as a pathological behavior of Pad\'e approximation.
\begin{theorem}\label{dense-poles}
Let $\MS=(m_{k},n_{k})_{k\geq 0}$ be a sequence of pairs of positive integers such that $(m_{k}+n_{k})_{k\geq 0}$ is unbounded. Let $\F^{\PP}_{\MS}$ (resp. 
$\F^{\ZZ}_{\MS}$) denote the subset of power series $f$ of $\F$ such that 
$f\in\NN_{m_{k},n_{k}}$ for all $k\geq 0$ and the set consisting of the poles (resp. zeroes) of all Pad\'e approximants $[f;m_{k}/n_{k}]$, $k\geq 0$, is dense in $\C$. 
Then,\\
(i) The set $\F^{\PP}_{\MS}$ is a 
dense $G_{\delta}$ (hence residual)
subset of $(\F,d)$.\\
(ii) The set $\F^{\ZZ}_{\MS}$ is also a 
dense $G_{\delta}$
subset of $(\F,d)$.
\end{theorem}
\begin{remark}
For the row sequence $(m,1)_{m\geq 0}$, an explicit series, with a positive radius of convergence and which satisfies assertion (i), is given in \cite[\S 45]{Per}. For the diagonal sequence $(n,n)_{n\geq 0}$, an example is also given in \cite{Wal}.
\end{remark}
\begin{proof}
%We directly prove assertion (ii) and first assume that 
We consider two cases:\\
{\bf First case: $(m_{k})_{k\geq 0}$ is unbounded.}
%Let $\FF_{k}\subset\F$, $k\geq 0$, be the subset of series $f$ such that $f\in\DD_{m_{k},n_{k}}$. 
We know from Proposition \ref{D-N} that $\NN_{k}$ is an open and dense subset of $\F$.
Next, let $(V_{j})_{j\geq 0}$ be a denumerable basis of open sets of $\C$ and set
\begin{equation}\label{def-Fj}
F_{\MS,j}=\{f\in\F,~\exists k\geq 0,~f\in\DD_{k}\text{ and }[f;m_{k}/n_{k}]\text{ has a pole in }V_{j}\}.
\end{equation}
Then 
$$\F^{\PP}_{\MS}=(\bigcap_{k\geq 0}\NN_{k})\cap(\bigcap_{j\geq 0}F_{\MS,j}),$$ 
and, by Baire's theorem, it suffices to show that each $F_{\MS,j}$ is open and dense. The fact that $F_{\MS,j}$ is open is clear since the Hankel determinant (\ref{Hankel}) and the denominator of the Pad\'e approximants in (\ref{deter}) are continuous functions of a finite number of coefficients of the power series. Let $\P$ be the set of polynomials and
$$\P_{\MS}=\{P\in\P,~\exists k\geq 0,~\deg P=m_{k}-1\}.$$
From the assumption that the sequence $(m_{k})_{k\geq 0}$ is unbounded, we know that $\P_{\MS}$ is dense in $(\F,d)$. To prove that $F_{\MS,j}$ is dense in $\F$, it is thus sufficient to show that, for a given $P\in\P_{\MS}$, there is a $f\in F_{\MS,j}$ close to $P$. Assume 
$$P(z)=a_{0}+\ldots+a_{m_{k}-1}z^{m_{k}-1},\qquad a_{m_{k}-1}\neq 0,$$
and consider the series (actually the polynomial)
\begin{equation}\label{def-f}
f(z)=P(z)+f_{m_{k}-1+n_{k}}z^{m_{k}-1+n_{k}}+f_{m_{k}+n_{k}}z^{m_{k}+n_{k}}\in\F,\end{equation}
where $f_{m_{k}-1+n_{k}}$ and $f_{m_{k}+n_{k}}$ are small and non vanishing coefficients.
We show that $f_{m_{k}-1+n_{k}}$ and $f_{m_{k}+n_{k}}$ can be chosen so that $f\in F_{\MS,j}$. Indeed, $f\in\DD_{k}$ because the determinant
$$C_{m_{k},n_{k}}=\begin{vmatrix}a_{m_{k}-n_{k}+1} & \ldots & a_{m_{k}-1} & 0\\
\vdots & \iddots & \iddots & \vdots\\
a_{m_{k}-1} & \iddots & 0 & 0\\
0 & \ldots & 0 & f_{m_{k}-1+n_{k}}\end{vmatrix}=f_{m_{k}-1+n_{k}}a_{m_{k}-1}^{n_{k}-1}$$
is non vanishing. Moreover, according to the second formula in (\ref{deter}), the denominator of the Pad\'e approximant $[f;m_{k}/n_{k}]$ equals
$$\begin{vmatrix}a_{m_{k}-n_{k}+1} & \ldots & a_{m_{k}-1} & 0 & 0\\
\vdots & \iddots & \iddots & \vdots & \vdots\\
a_{m_{k}-1} & \iddots & 0 & 0 & f_{m_{k}-1+n_{k}}\\
0 & \ldots & 0 & f_{m_{k}-1+n_{k}} & f_{m_{k}+n_{k}}\\
z^{n_{k}} & \ldots & \ldots & z & 1
\end{vmatrix}.$$
Expanding this determinant along its last column, we readily obtain the expression
\begin{equation}\label{Pade-den}
-f_{m_{k}-1+n_{k}}^{2}P_{n_{k}-2}(z)z^{2}
-f_{m_{k}+n_{k}}a_{m_{k}-1}^{n_{k}-1}z
+f_{m_{k}-1+n_{k}}a_{m_{k}-1}^{n_{k}-1},
\end{equation}
where $P_{n_{k}-2}(z)$ is some polynomial of degree $n_{k}-2$ which depends only on the coefficients of $P(z)$. Note that $P_{n_{k}-2}(z)=0$ if $n_{k}=1$. It is clear that the nonzero coefficients $f_{m_{k}-1+n_{k}}$ and $f_{m_{k}+n_{k}}$ can be chosen as small as desired and such that (\ref{Pade-den}) vanishes at a given point (distinct from the origin) of $V_{j}$. This shows that $F_{\MS,j}$ is dense in $\F$ and finishes the proof of assertion (i) when $(m_{k})_{k\geq 0}$ is unbounded.

The above reasoning can be repeated for the set of zeroes of the Pad\'e approximants. Indeed, using the first formula in 
(\ref{deter}), which gives the numerator of the Pad\'e approximant, one just has to replace (\ref{Pade-den}) with
\begin{equation}\label{Pade-num}
-f_{m_{k}-1+n_{k}}^{2}R_{m_{k}-2}(z)z^{2}
-f_{m_{k}+n_{k}}a_{m_{k}-1}^{n_{k}-1}P(z)
+f_{m_{k}-1+n_{k}}a_{m_{k}-1}^{n_{k}-1}\tilde P(z),
\end{equation}
where $R_{m_{k}-2}(z)$ is some polynomial of degree $m_{k}-2$ which depends only on the coefficients of $P(z)$, and $\tilde P(z)=P(z)$ if $n_{k}>1$ and $\tilde P(z)=P(z)+f_{m_{k}}z^{m_{k}}$ if $n_{k}=1$. One can always find a point $z_{j}$ in $V_{j}$ such that $P(z_{j})\neq 0$. Hence, two nonzero coefficients $f_{m_{k}-1+n_{k}}$ and $f_{m_{k}+n_{k}}$ can be chosen, sufficiently small and such that (\ref{Pade-num}) vanishes at $z_{j}\in V_{j}$. This shows that the set of series having a Pad\'e approximant with a zero in $V_{j}$ is dense in $\F$. 

{\bf Second case: $(m_{k})_{k\geq 0}$ is bounded.} Then $(n_{k})_{k\geq 0}$ is unbounded. We denote by $\tilde\MS$ the sequence $(n_{k},m_{k})_{k\geq 0}$ and by $\F_{0}$ the subset of $\F$ of power series with non vanishing constant coefficient $a_{0}$. It is an open and dense subset of $\F$. We also set
$$%\begin{equation}\label{def-Fj}
\tilde F_{\MS,j}=\{f\in\F,~\exists k\geq 0,~f\in\DD_{m_{k},n_{k}}\text{ and }[f;m_{k}/n_{k}]\text{ has a zero in }V_{j}\}.
%\end{equation}.
$$ 
It is an open subset of $\F$. To show it is dense, it suffices to check that $\F_{0}\cap \tilde F_{\MS,j}$ is dense in $\F_{0}$. This follows from the facts that \\
(i)
$\F_{0}\cap\tilde F_{\MS,j}=(\F_{0}\cap F_{\tilde\MS,j})^{-1}:=\{1/f,~f\in \F_{0}\cap F_{\tilde\MS,j}\}$, \\
(ii)
$\F_{0}\cap F_{\tilde\MS,j}$ is dense in $\F_{0}$, \\
(iii)
the map which takes the inverse of a function is a homeomorphism in $\F_{0}$. 
\\
The first fact is a consequence of Proposition \ref{inv-Pade}. The second fact follows from the first part of the proof, where we note, in the definition (\ref{def-f}) of $f$, that $P\in\F_{0}$ implies $f\in\F_{0}$. For the last fact, we recall that the coefficients $(b_{n})_{n\geq 0}$ of the reciprocal of a series with coefficients $(a_{n})_{n\geq 0}$ are given by the recursive formulae
$$b_{0}=a_{0}^{-1},\quad b_{n}=-a_{0}^{-1}\sum_{i=1}^{n}a_{i}b_{n-i},\quad n\geq 1.$$
Since each $\tilde F_{\MS,j}$ is open and dense, we derive that $\F^{\ZZ}_{\MS}$ is a dense $G_{\delta}$ subset of $(\F,d)$. The fact that $\F^{\PP}_{\MS}$ is a dense $G_{\delta}$ subset is proved in a similar way.
The assertions of the theorem
%for the case of the space $(\F,d)$ 
are thus obtained.
\end{proof}
The functions in $H(\Omega)$ which have a reciprocal in $H(\Omega)$ are exactly the non vanishing ones. Such functions are not dense in $H(\Omega)$. Hence, in this space, we have to assume that the sequence $\MS$ of indices is such that $(m_{k})_{k\geq 0}$ is unbounded.
\begin{theorem}\label{dense-poles-ana}
Let $\MS=(m_{k},n_{k})_{k\geq 0}$ be a sequence of pairs of positive integers such that $(m_{k})_{k\geq 0}$ is unbounded.
The assertions of Theorem \ref{dense-poles} remain true in the space $H(\Omega)$ endowed with the topology of uniform convergence on compact subsets of $\Omega$.
\end{theorem}
\begin{proof}
The proof of Theorem \ref{dense-poles}, except for the second part where it is assumed that $(m_{k})_{k\geq 0}$ is bounded, can be repeated.
Indeed, $\NN_{k}$ and $F_{\SS,j}$ are open in $H(\Omega)$. Moreover,
from Runge's theorem, we know that $\P$ is dense in $H(\Omega)$ and the same holds true for $\P_{\MS}$. The sequel of the proof, showing that $F_{\SS,j}$ is dense in $H(\Omega)$, remains unchanged, where we remark that $d_{\infty}(P,f)$ can be made as small as we want by choosing the coefficients $f_{m_{k}-1+n_{k}}$ and $f_{m_{k}+n_{k}}$ sufficiently small.
\end{proof}
The space $\F$ can also be endowed with the product topology, where each copy of $\C$ is given the discrete topology. This topology is equivalently defined from the distance
$$\tilde d((a_{k})_{k},(b_{k})_{k})=2^{-j}\text{ with }\left\{
\begin{array}{ll} j=\inf\{k\geq 0,~a_{k}\neq b_{k}\}\text{ if }(a_{k})_{k}\neq (b_{k})_{k},
\\[10pt]
j=\infty\text{ if }(a_{k})_{k}=(b_{k})_{k}.\end{array}\right.$$
The space $(\F,\tilde d)$ is a complete metric space, whose topology is stronger than that of the distance $d$. Note that in a given neighbourhood of a series $S(z)=\sum_{k\geq 0}a_{k}z^{k}$, all series share with $S$ a certain number of its first coefficients $a_{0},a_{1},\ldots$. As in $H(\Omega)$, the subset of series which admit a reciprocal, i.e. with a nonzero constant coefficient, is not dense. Moreover, the sets $\NN_{k}$ in the proof of Theorem \ref{dense-poles} are still open in $(\F,\tilde d)$ but no more dense.
Hence, we only have a version of Theorem \ref{dense-poles-ana} in $(\F, \tilde d)$ which holds for subsequences, namely:
\begin{theorem}
Let $\MS=(m_{k},n_{k})_{k\geq 0}$ be a sequence of pairs of positive integers such that $(m_{k})_{k\geq 0}$ is unbounded.
%There exists an infinite sequence $I\subset\N$ and a power series $f\in\F$ such that $f\in\DD_{m_{k},n_{k}}$ for all $k\in I$ and the poles of the Pad\'e approximants $[f;m_{k}/n_{k}]$, $k\in I$, are dense in $\C$. 
The subset $\hat\F_{\MS}^{\PP}$ (resp. $\hat\F_{\MS}^{\ZZ}$) of power series $f$ of $\F$ such that there exists an infinite sequence $I\subset\N$ with $f\in\DD_{k}$ for all $k\in I$ and the poles (resp. zeroes) of the Pad\'e approximants $[f;m_{k}/n_{k}]$, $k\in I$, are dense in $\C$,
is a %residual
dense $G_{\delta}$ 
subset of $(\F,\tilde d)$.
\end{theorem}
\begin{proof}
Now, we have $\hat\F_{\MS}^{\PP}=\cap_{j\geq 0}F_{\SS,j}$ where $F_{\SS,j}$ is defined as in (\ref{def-Fj}).
Each set $F_{\SS,j}$ is open with respect to $\tilde d$ because the conditions for a series $f$ to belong to $F_{\SS,j}$ only involve a finite number of its coefficients. To show that $F_{\SS,j}$ is dense in $\F$, we consider a given series $f_{0}$ in $\F$ and first pick a polynomial $P\in\P_{\MS}$ sufficiently close to it, namely a truncation of $f_{0}$ of degree $m_{k}-1$ so that $\tilde d(f_{0},P)=2^{-m_{k}}$ is small enough, and then construct $f$ as in the proof of Theorem \ref{dense-poles}. It still shows the density of $F_{\SS,j}$ because $\tilde d(P,f)=2^{-(m_{k}-1+n_{k})}\leq 2^{-m_{k}}$ is also small. The assertion about $\hat\F_{\MS}^{\ZZ}$ is also proved as in Theorem \ref{dense-poles}.
\end{proof}
The above results show that the usual behavior of Pad\'e approximants is, in a way, wild.
In an opposite direction, and in connection with Baker's conjecture, a generic convergence result for the Pad\'e approximants of entire functions was obtained in \cite{BOR}. It was recently extended to arbitrary domains and also in other ways in \cite{FNes}. Here, we state a version for a simply connected domain $\Omega$ which is slightly more precise than {\cite[Theorem 3.7]{FNes}}.
\begin{theorem}\label{F-N}
Let $(m_{k},n_{k})_{k\geq 0}$ be a sequence of pairs of positive integers such that $(m_{k})_{k\geq 0}$ is unbounded. 
The subset $\hat H(\Omega)$ of functions $f$ in $H(\Omega)$ such that there exists an infinite sequence $I\subset\N$ with $f\in\NN_{k}$ for all $k\in I$ and the Pad\'e approximants $[f;m_{k}/n_{k}]$, $k\in I$, tend to $f$ in $H(\Omega)$ as $k\to\infty$,
is a %residual
dense $G_{\delta}$ 
subset of $H(\Omega)$.
\end{theorem}

\begin{proof}\cite[Theorem 3.7]{FNes} tells that the conclusion holds true with $\DD_{k}$ instead of $\NN_k$. Since $\NN_{k}$ is open and dense in $H(\D)$, the result follows by Baire Category Theorem.
\end{proof}

Note that the assumption that $(m_{k})_{k\geq 0}$ is unbounded cannot be discarded in the above theorem. Indeed, if the sequence $(m_{k})_{k\geq 0}$ is bounded, then, by Rouch\'e's theorem, the functions $f\in\hat H(\Omega)$, uniform limit of Pad\'e approximants $[f;m_{k}/n_{k}]$, can only have a number of zeroes bounded by $\sup_{k}m_{k}$ in $\Omega$. The subset of such functions $f$ is not dense in $H(\Omega)$.

When $\Omega=\C$, the Pad\'e approximants of degrees $(m_{k},n_{k})$, $k\in I$, of the functions in the previous theorem have poles which all go to infinity as $k\to\infty$. Hence the subset of such functions is residual in $H(\C)$. More generally, we have the following assertion.
\begin{theorem}\label{poles-inf}
Let $(m_{k},n_{k})_{k\geq 0}$ be a sequence of pairs of positive integers such that $(m_{k})_{k\geq 0}$ is unbounded. 
The subset  $\tilde H(\Omega)\subset H(\Omega)$ of functions $f$ such that there exists an infinite sequence $I\subset\N$ with $f\in\NN_{k}$ for all $k\in I$ and the poles of the Pad\'e approximants $[f;m_{k}/n_{k}]$, $k\in I$, tend to infinity as $k\to\infty$,
is a %residual
dense $G_{\delta}$ 
subset of $H(\Omega)$. The same assertion holds true in the space of power series $\F$ endowed with the topology of the distance~$d$.
\end{theorem}
\begin{remark}
We do not know if the theorem remains true when the sequence $(m_{k})_{k\geq 0}$ is bounded (and $(n_{k})_{k\geq 0}$ is unbounded). Actually, it cannot be true if we consider functions such that the poles of the Pad\'e approximants of the {\it entire} sequence $(m_{k},n_{k})_{k\geq 0}$ tend to infinity. Indeed, by a result of Gonchar, see \cite[\S 3]{Gon}, if $m\in\N$ is a given integer, $f$ is a function analytic e.g. in a disk $\D_{r}$, and the poles of the Pad\'e approximants $[f;m/k]$, $k\geq k_{0}$, lie outside of this disk, then the Pad\'e approximants tend locally uniformly to $f$ in $\D_{r}$. This entails that $f$ has at most $m$ zeroes in $\D_{r}$ and the subset of such functions is not dense in $H(\D_{r})$. 
\end{remark}
\begin{proof}[Proof of Theorem \ref{poles-inf}]
Let $\D_{j}$ be the open disk centered at 0, of radius $j$ and let
$$H_{j,k}=\{f\in H(\Omega),~f\in\NN_{k}\text{ and }
[f;m_{k}/n_{k}]\text{ has no poles in }\bar\D_{j}\}.$$
Then
$$\tilde H(\Omega)=\bigcap_{j=1}^{\infty}\bigcup_{k=1}^{\infty}H_{j,k}.$$
Each set $H_{j,k}$ is clearly open in $H(\Omega)$. Hence, $\tilde H(\Omega)$ is a $G_{\delta}$ subset of $H(\Omega)$. By Theorem \ref{F-N}, $\hat H(\C)$ is dense in $H(\C)$ which is itself a dense subset of $H(\Omega)$ (endowed with its topology). Since the topology of $H(\C)$ is finer than the topology of $H(\Omega)$, we deduce that $\hat H(\C)$ is dense in $H(\Omega)$. Consequently  $H_{j,k}$, which contains $\hat H(\C)$ is also dense in $H(\Omega)$. By Baire's theorem, $\tilde H(\Omega)$ is dense in $H(\Omega)$ as well. 

When $\Omega\neq\C$, we display another proof of the density of $\cup_{k=1}^{\infty}H_{j,k}$, $j\geq 1$, which does not use Theorem \ref{F-N}. Let $P$ be a polynomial and $K$ a compact subset of $\Omega$. We show that there exists a function $f$ in the previous union which is close to $P$ on $K$. 
%First, we assume that the sequence $n_{k}$, $k\geq 0$, is unbounded. 
Let $m_{k}\geq\deg P$, which is possible since $m_{k}$ is unbounded.
Since $\Omega$ is simply connected, and different from $\C$, it cannot contain a neighborhood of infinity. Hence, there exists some complex point $\mu$, large enough, so that $\mu\notin \Omega\cup\bar\D_{j}$, $P(\mu)\neq 0$, and $Q_{n_{k}}(z)=(1-z/\mu)^{n_{k}}$ is uniformly close to 1 on $K$.
%Then, by applying Runge's theorem on a sequence of increasing compact sets whose union is $\C$, there exists a sequence of polynomials $Q_{n_{k}}$ which tend uniformly to the constant 1 on $K$ and whose poles tend to infinity, as $k\to \infty$. 
Now, we choose
%$$f(z)=\frac{P(z)+\lambda z^{m_{k}}}{Q_{n_{k}}(z)}\in H(\Omega),$$
%with $m_{k}\geq\deg P$, which is possible since $m_{k}$ tends to infinity, 
%%$Q_{n_{k}}$ sufficiently close to 1 on $K$ and with no roots in $\bar\D_{j}$, 
%and $\lambda$ sufficiently small. 
$$f(z)=P(z)/Q_{n_{k}}(z)\in H(\Omega).$$
Then, $f$ is close to $P$ on $K$. Moreover, 
%$\lambda$ can be chosen so that $f\in\DD_{k}$. Also 
$[f;m_{k}/n_{k}]$ exists and equals $f$, the Pad\'e denominator $Q_{m_{k},n_{k}}(f)$ equals $Q_{n_{k}}$ which is of degree $n_{k}$, and $f$ has no poles in $\bar\D_{j}$.
%Second, assume that the sequence $n_{k}$, $k\geq 0$, is bounded.
%In this case, we choose
%$$f(z)=\frac{P(z)Q_{n_{k}}(z)+\lambda z^{m_{k}}}{Q_{n_{k}}(z)},\qquad 
%Q_{n_{k}}(z)=(1-z/\mu)^{n_{k}},$$
%with $\mu\notin \Omega\cup\bar\D_{j}$, and $m_{k}\geq\deg P+n_{k}$, which is possible since the sequence $m_{k}$, $k\geq 0$, is unbounded. Then, $f\in H(\Omega)$, and for some $\lambda$ small enough, $f$ is close to $P$ on $K$ and $f\in\DD_{k}$. Moreover, $[f;m_{k}/n_{k}]=f$ has no poles in $\bar\D_{j}$.

The above proofs can be repeated in the space $(\F,d)$.
\end{proof}

\section{A particular set of Pad\'e universal series and applications}\label{part-set}
%In this section, ...

%As mentioned in the introduction our results hold for Pad\'e universal series associated with the Taylor expansion at $0$ of a holomorphic function on a simply connected domain $\Omega$ containing $0$. 
For the sake of clarity, we restrict ourselves in the sequel to the case $\Omega =\D$.
Nevertheless, our results would hold in a simply connected domain $\Omega$ containing the origin. 

\subsection{Pad\'e universal series whose Pad\'e approximants have prescribed poles}

We first exhibit Pad\'e universal series admitting a sequence of Pad\'e approximants of degrees $(\lambda_{n},1)$, $\lambda_{n}\to\infty$, whose denominators equal the polynomial
$1-z$.
\begin{theorem}\label{main-thm}
Let $\mu=\left(p_n\right)_n$ be an unbounded sequence of positive integers. There exists a function $f$%$=\sum _{n=0}^{\infty}a_nz^n$ 
in $H(\D)$ such that for every compact set $K\subset \C \setminus (\D\cup\{1\})$ with $\C \setminus K$ connected and every function $h\in A(K)$, there is a subsequence $\left(\lambda_n\right)_n$ of $\mu$ with the following properties:
\begin{enumerate}[a)]
\item $f \in \NN_{\lambda_n,1}$ for all $n=1,2,\ldots$;
\item $[f;\lambda_n/1]\rightarrow h$ in $A(K)$ as $n\rightarrow \infty$;
\item $[f;\lambda_n/1]\rightarrow f$ in $H(\D)$ as $n\rightarrow \infty$;
\item $Q_{\lambda_n,1}(f)(z)=1-z$.
\end{enumerate}
The set $\UU_{1-z}^{\mu}(\D)$ of such functions is a dense meagre subset of $H(\D)$.
\end{theorem}

The proof is based on two lemmas.
\begin{lemma}\label{lemma-geo} 
Let $a\in\C\setminus\D$. There exists a sequence $\left(K_n\right)_n$ of compact subsets of $\C\setminus (\D\cup\{a\})$, with $\C \setminus K_n$ connected for all $n$, such that every compact set $K\subset \C\setminus (\D\cup\{a\})$ with $\C \setminus K$ connected is included in a set $K_n$ for some $n$.
\end{lemma}
\begin{proof}The proof is a slight modification of that of \cite[Lemma 2.1]{Nes}. For an integer $k\geq 1$, let $\left(\Gamma _j^k\right)_{j}$ be an enumeration of all simple polygonal lines from $0$ to $a$, passing through $k$, having a finite number of vertices, all of them with rational coordinates. Let $\left(K_n\right)_{n}$ be an enumeration of the family
$$\left(\left\{z\in \C,\,1\leq |z|\leq k,\,\text{dist}\left(z,\Gamma _j^k\right)\geq l^{-1}\right\}\right)_{j,k,l\geq 1}.$$
Then $\left(K_n\right)_n$ satisfies the assertion of the lemma. Indeed, for any compact set $K$ in the complement of $\D\cup\{a\}$, with connected complement, there exists a simple polygonal line with finitely many vertices of rational coordinates connecting $0$ to $a$, through $k$, $k$ large enough, whose distance to $K$ is strictly positive. Hence $K$ is included in $K_n$ for some $n$.
\end{proof}
\begin{remark}\rm{Note that the sequence $\left(K_n\right)_n$ in Lemma \ref{lemma-geo} is not an increasing sequence.}
\end{remark}
The next lemma follows easily from Mergelyan's theorem.
% (see \cite[Lemma 5]{Bay}).
\begin{lemma}\label{mergelyan}Let $a\in\C\setminus\D$. Let $K$ be a compact subset of $\C \setminus (\D\cup\{a\})$ with connected complement, $L$ a compact subset of $\D$ and $h\in A(K)$. Then for every $\varepsilon >0$, every $\alpha\in \C$ and every integer $p\geq 1$, there exists a polynomial $P=\sum _{k=p}^qa_kz^k$ such that
\begin{enumerate}[a)]
\item $P(a)\neq 0$ and $P(a)\neq \alpha$.
\item $\left\Vert P-h \right\Vert _K \leq \varepsilon$;
\item $\left\Vert P\right\Vert _L \leq \varepsilon$.
\end{enumerate}
\end{lemma}
\begin{proof}%\cite[Lemma 5]{Bay} gives a polynomial $P$ satisfying b) and c). 
Since $K\cup\D$ has connected complement, one may find a polynomial $Q$ such that
$$\|Q\|_{\D}\leq\epsilon,\quad\|Q-z^{-p}h\|_{K}\leq\epsilon.$$
The polynomial $P=z^{p}Q$ satisfies b) and c).
If a) is not satisfied, it suffices to add to $P$ a monomial of degree greater than $p$ with a coefficient small enough.
\end{proof}
\begin{proof}[Proof of Theorem \ref{main-thm}]We directly prove that $\UU_{1-z}^{\mu}(\D)$ is a dense meagre subset of $H(\D)$.  
%To prove that this set is meagre we first claim the following. 
It is easy to see that
%\begin{claim}\rm{
a power series $\sum _k a_kz^k$ whose Pad\'e approximant of degrees $(p,1)$ has (non reducible) denominator $1-z$ must satisfy $a_{p+1}=a_p$.
%}
%\end{claim}
%\begin{proof}[proof of the claim]Since $P=\sum _ kb_kz^k$ is a polynomial of degree at most $p$, the $p$-th and $(p+1)$-th coefficients in the Taylor expansion of $P/(1-z)$ coincide (they are both given by $\sum _{k=0}^{p}b_k$). Then $a_p$ and $a_{p+1}$ have to be equal.
%
%We could also make use of the second formula of \ref{deter}.
%\end{proof}
Hence, the complement of $\UU_{1-z}^{\mu}(\D)$ contains the subset $E$ of $H(\D)$ consisting of those power series $\sum _k a_kz^k$ with radius of convergence at least $1$ such that for any $p\in \N$, $a_p\neq a_{p+1}$. The set $E$ can be written as $E=\bigcap_{n\geq 0}E_n$ where
$$E_n=\left\{\sum _ka_kz^k \in H(\D),\,a_k\neq a_{k+1}\text{ for every }k=0,\ldots ,n\right\}.$$
Each $E_n$, $n\in \N$, is clearly open and dense in $H(\D)$, then by the Baire Category Theorem, $E$ is a dense $G_{\delta}$ subset of the Baire space $H(\D)$. Thus the complement of $\UU_{1-z}^{\mu}(\D)$ is residual and $\UU_{1-z}^{\mu}(\D)$ is meagre.

We now turn to proof of the density of $\UU_{1-z}^{\mu}(\D)$. We fix a polynomial $T$ in $H(\D)$, a compact set $L\subset \D$ and $0<\varepsilon _0<1$. Let \\
-- $\left(L_n\right)_n$ be an exaustion of compact subsets of $\D$ such that for all $n\geq 0$, $L\subset L_n$, \\
-- $\left(K_n\right)_n$, a sequence of compact sets given by Lemma \ref{lemma-geo} with $a=1$,\\
-- $\left(Q_n\right)_n$, an enumeration of polynomials with coefficients in $\Q+i\Q$,\\
-- $\phi,\,\psi:\N\rightarrow \N$, two functions such that, for every pair $(m,l)$ of positive integers, there exist infinitely many integers $n$ such that $(\phi(n),\psi(n))=(m,l)$.

We build by induction a power series $f$ in $\UU^{\mu}(\D)$ which is close to $T$ in $H(\D)$, namely %up to $\varepsilon _0$ in the sense that
$$\left\Vert f-T\right\Vert _L \leq \varepsilon _0,$$
where $\varepsilon _0>0$.
We set $f_0=(1-z)T$. Then we assume that the polynomial $f_j$ has been built for some $j\geq 1$. We set $f_{j+1}=f_j+P$ where $P$ is given by Lemma \ref{mergelyan} with
$$a=1,\qquad K=K_{\psi(j+1)},\qquad L=L_{j+1},\qquad h=(1-z)Q_{\phi(j+1)}-f_j,$$
and
$$\varepsilon =%\frac{
\varepsilon _0 2^{-j-1}\min(1,d(1,K_{\psi(j+1)}),d(1,L)),\quad
\alpha=-f_{j}(1),\quad p=\min_{n}\left\{p_n;\,p_n\geq \text{deg}(f_{j})\right\}+2.
$$
%\begin{itemize}\item $\varepsilon =\frac{\varepsilon _0}{2^{j+1}}\min(1,d(1,K_{\psi(j+1)}),d(1,L))$,
%\item $a=-f_{j}(1)$,
%\item $p=\min\left\{p_n;\,p_n\geq \text{deg}(f_{j})\right\}+2$,
%\item $K=K_{\psi(j+1)}$, $L=L_{j+1}$
%\item $h=(1-z)Q_{\phi(j+1)}-f_j$.
%\end{itemize}
By construction we have, for every $j\geq 1$,
\begin{enumerate}[(i)]
\item $\displaystyle{\left\Vert (1-z)Q_{\phi(j)}-f_j\right\Vert _{K_{\psi(j)}}\leq \varepsilon \leq 2^{-j}\min(1,d(1,K_{\psi(j)}))}$;
\item $\displaystyle\left\Vert f_{j+1}-f_j\right\Vert _{L_{j+1}} \leq \varepsilon 
\leq 2^{-j-1}d(1,L)\varepsilon _0$.
\end{enumerate}
We now define 
\begin{equation}\label{def-tildef}
\tilde{f}=\sum _{j\geq 0}\left(f_{j+1}-f_j\right)+f_0.
\end{equation}
By (ii) above, we deduce that $\tilde{f}\in H(\D)$ and that $f_N$ tends to $\tilde{f}$ in $H(\D)$, as $N$ tends to $\infty$. Moreover
$$\tilde{f}-(1-z)T=\sum _{j\geq 0}(f_{j+1}-f_j),$$
so that (ii) also implies
$$\|\tilde{f}-(1-z)T\|_L \leq d(1,L)\varepsilon _0,$$
since $L\subset L_j$ for every $j\geq 1$. Finally we define $f=\tilde{f}/(1-z)$ and show that $f$ gives the desired power series. From the above observations, it is clear that $f\in H(\D)$ and that
$$\left\Vert f-T\right\Vert _L \leq \varepsilon _0.$$
It remains to show that $f\in \UU _{1-z}^{\mu}(\D)$. Notice that, by construction, for every $j\geq 1$, if we denote by $p_{n_j}$ the smallest element of $\mu$ such that $p_{n_j}\geq \text{deg}\left(f_j\right)$, then
\begin{equation}\label{eq-thm-pole-1}f\in \NN_{p_{n_j},1}\quad\text{and}\quad[f;p_{n_j}/1]=\frac{f_j}{1-z}.
\end{equation}
Indeed, since for every $j\geq 1$, $f_{j+1}=f_{j}+P$ with $\text{val}(P)\geq p_{n_j}+2$, the $p_{n_j}+1$ first coefficients of the Taylor expansions of $f$ and $f_j/(1-z)$ coincide. In addition $f_j$ does not vanish at $1$ since $f_j=f_{j-1}+P$ with $P(1)\neq -f_{j-1}(1)$. As $\text{deg}\left(f_j\right) \leq p_{n_j}$ (\ref{eq-thm-pole-1}) follows.
%and PROPOSITION2.2.

We now fix $\varepsilon >0$, a compact set $K\subset \C\setminus (\D\cup \{1\})$ with connected complement and a function $h\in A(K)$. Let $(l,m)$ be two integers such that $K\subset K_m$ and such that $\left\Vert Q_l-h\right\Vert _{K_m}\leq \varepsilon /2$. Let then $\left(v_j\right)_j$ be an infinite sequence such that $(l,m)=(\phi(v_j),\psi(v_j))$ for every $j\geq 0$, and consider the subsequence $\left(p_{n_j}\right)_j$ of $\mu$ where $p_{n_j}$ is the smallest element $p_n$ of $\mu$ with $p_{n}\geq \text{deg}\left(f_{v_j}\right)$. By (i) and the above, there exists $j$ large enough such that
$$\left\Vert [f;p_{n_j}/1]-h\right\Vert _K\leq  \bigg\|\frac{f_{v_j}}{1-z}-h\bigg\|_{K_m}\leq  \left\Vert \frac{f_{v_j}}{1-z}-Q_l\right\Vert _{K_m}+ \left\Vert Q_l-h\right\Vert _{K_m}\leq  \frac{\varepsilon }{2}+\frac{\varepsilon }{2}.
$$
Property c) in the theorem follows from (\ref{eq-thm-pole-1}) and the fact that $f_{v_j}/(1-z)$ tends to $f$ in $H(\D)$.
\end{proof}

Theorem \ref{main-thm} holds true with $1-z$ replaced by $1-z/w$, for any point $w\in \C \setminus \D$. Actually, one can replace the polynomial $1-z$ by any polynomial $Q$ of some degree $q\geq 1$,
\begin{equation}\label{def-Q}
Q(z)=\prod _{i=1}^q \left(1-\frac{z}{w_i}\right)\text{ with }|w_i|\geq 1\text{ for every }i,
\end{equation}
and the assertion $f\in \NN_{\lambda _n,1}$ by the assertion $f\in \NN_{\lambda _n,q}$. This is the content of the following theorem, which is the main result of this section.

\begin{theorem}\label{main-thm-Q}
Let $W=(w_1,\ldots,w_q)$ be a family of $q$ points in $\C \setminus \D$, not necessarily distinct. Let $\mu=\left(p_n\right)_n$ be an unbounded sequence of integers and $Q$ a polynomial of degree $q$ as in (\ref{def-Q}).
There exists a function $f$%$=\sum _{n=0}^{\infty}a_nz^n$ 
in $H(\D)$ such that for every compact set $K\subset \C \setminus (\D\cup W)$ with $\C \setminus K$ connected, and every function $h\in A(K)$, there is a subsequence $\left(\lambda_n\right)_n$ of $\mu$ with the following properties:
\begin{enumerate}[a)]
\item $f \in \NN_{\lambda_n,q}$ for all $n=1,2,\ldots$;
\item $[f;\lambda_n/q]\rightarrow h$ in $A(K)$ as $n\rightarrow \infty$;
\item $[f;\lambda_n/q]\rightarrow f$ in $H(\D)$ as $n\rightarrow \infty$;
\item $Q_{\lambda_n,q}(f)=Q$.
\end{enumerate}
The set $\UU_{Q}^{\mu}(\D)$ of such functions is a dense meagre subset of $H(\D)$.
\end{theorem}

The proof is an easy modification of that of Theorem \ref{main-thm}, based on the two following lemmas whose proofs are similar to those of Lemmas \ref{lemma-geo} and \ref{mergelyan}, respectively, and are omitted.
\begin{lemma}\label{lemma-geo-Q}Let $W=(w_1,\ldots,w_q)$ be a family of $q$ points in $\C \setminus \D$, not necessarily distinct. There exists a sequence $\left(K_n\right)_n$ of compact subsets of $\C\setminus (\D\cup W)$, with $\C \setminus K_n$ connected for all $n$, such that every compact set $K\subset \C\setminus (\D\cup W)$ with $\C \setminus K$ connected is included in a set $K_n$ for some $n$.
\end{lemma}

\begin{lemma}\label{mergelyan-Q}Let $W=(w_1,\ldots,w_q)$ be a family of $q$ points in $\C \setminus \D$, not necessarily distinct. Let $K$ be a compact subset of $\C \setminus (\D\cup W)$ with connected complement and $L$ a compact subset of $\D$. Let also $h\in A(K)$. Then for every $\varepsilon >0$, every family $(a_1,\ldots ,a_q)\in \C$ and every integer $p\geq 1$, there exists a polynomial $P=\sum _{k=p}^rb_kz^k$ such that
\begin{enumerate}[a)]
\item $P(w_i)\neq 0$ and $P(w_i)\neq a_i$ for every $1\leq i \leq q$;
\item $\left\Vert P-h \right\Vert _K \leq \varepsilon$;
\item $\left\Vert P\right\Vert _L \leq \varepsilon$.
\end{enumerate}
\end{lemma}
\begin{proof}[Proof of Theorem \ref{main-thm-Q}] It suffices to repeat the proof of Theorem \ref{main-thm}. The set $\UU_{Q}^{\mu}(\D)$ being meager is still a consequence of the fact that the denominator of the $\left(\lambda_n,q\right)$ Pad\'e approximants of any $f\in\UU_{Q}^{\mu}(\D)$ equals the prescribed polynomial $Q$. The density follows by applying Lemma \ref{lemma-geo-Q} and defining the recurrence with $f_0=Q(z)T$ and then, at the $(j+1)$-th step, by using Lemma \ref{mergelyan-Q} with
$$K=K_{\psi(j+1)}, \quad L=L_{j+1},\quad h=Q\cdot Q_{\phi(j+1)}-f_j,$$
$$\varepsilon =\varepsilon _0 2^{-j-1}\min(1,d(W,K_{\psi(j+1)}),d(W,L)),
\quad a_i=-f_{j}(w_i)\text{ for }1\leq i \leq q,$$
and
$$
p=\min_{n}\left\{p_n;\,p_n\geq \text{deg}(f_{j})\right\}+q+1,$$
where the notations in the proof of Theorem \ref{main-thm}) are used.
\end{proof}

%\begin{remark}A similar result for a polynomial $Q$ with roots in $\D\setminus \{0\}$ can be obtained if we restrict ourselves to the existence of a universal power series with radius of convergence at most $\min\{|w_1|,\ldots,|w_q|\}$.
%\end{remark}

We end this paragraph with a result which will be useful in the next section. 
%Let us introduce the following definition.

\begin{definition}\label{defi-new}Let $W=(w_1,\ldots,w_q)$ be a family of $q$ points in $\C \setminus \D$, not necessarily distinct. Let $\mu$ be an unbounded sequence of positive integers and $Q$ be a polynomial 
as in (\ref{def-Q}).
%given by
%$$Q(z)=\prod _{i=1}^q \left(1-\frac{z}{w_i}\right)\text{ with }|w_i|\geq 1\text{ for every }i.$$
Let also $K\subset \C \setminus (\D\cup W)$ be a compact subset with $\C \setminus K$ connected. We denote by $\UU_{Q}^{\mu}(\D,K)$ the set of functions $f \in H(\D)$ satisfying the same properties as those in $\UU _Q^{\mu}(\D)$, except that
%(see Theorem \ref{main-thm-Q}) 
Property b) is verified only on the compact set $K$.
\end{definition}

%Writing $MS=\left(p_n,q\right)_n$ with $\left(p_n\right)_n=\mu$ and if $Q$ is any polynomial of degree $q$ with zeros outside $K\cup \D$, it is clear that the following inclusion holds:
%$$\UU_{Q}^{\mu}(\D,K) \cup \UU ^{S}(\D) \subset \UU^{[\mu,q]}(\D,K)\text{ (See Theorem \ref{thm-main-intro} for the definition of }\UU ^{S}(\D)\text{)}.$$
%Following the proof of Theorem \ref{main-thm-Q} (see the beginning of the proof of Theorem \ref{main-thm}) (resp. using Theorem \ref{DFN-main}), we easily get:

The next corollary is a slight refinement of Theorem \ref{main-thm-Q}.

\begin{corollary}\label{last-coro-gene}Let $W=(w_1,\ldots,w_q)$ be a family of $q$ points in $\C \setminus \D$, not necessarily distinct. Let $\mu^i=\left(p^i_n\right)_n$, $i\in \N$, be a countable family of unbounded sequences of positive integers and $Q$ be a polynomial 
as in (\ref{def-Q}).
% and $Q$ a polynomial given by
%$$Q(z)=\prod _{i=1}^q \left(1-\frac{z}{w_i}\right)\text{ with }|w_i|\geq 1\text{ for every }i.$$
Let also $\left(K_i\right)_i$, $i\in \N$, be a countable family of compact sets of $\C \setminus (\D\cup W)$ with $\C \setminus K_i$ connected for every $i\in \N$. Then the set
$$\bigcap _{i\in \N}\UU_{Q}^{\mu^i}(\D,K_i)$$
is a dense meagre subset of $H(\D)$.
\end{corollary}

\begin{proof}We only give a scheme of the proof. The set is meagre in $H(\D)$ as a subset of meagre sets. For the density we let the reader check that it follows as in the proof of Theorem \ref{main-thm-Q} by building a universal Pad\'e series by induction, starting with $f_0=Q(z)T$ and using Lemma \ref{mergelyan-Q} for the $(j+1)$-th step with
$$K=K_{\psi(j+1)}, \quad L=L_{j+1},\quad h=Q\cdot Q_{\phi(j+1)}-f_j,$$
$$\varepsilon =\varepsilon _0 2^{-j-1}\min(1,d(W,K_{\psi(j+1)}),d(W,L)),
\quad a_i=-f_{j}(w_i)\text{ for }1\leq i \leq q,$$
and
$$
p=\min_{n}\left\{p_n^{\psi(j+1)};\,p_n^{\psi(j+1)}\geq \text{deg}(f_{j})\right\}+q+1.$$
%\begin{itemize}
%\item $\varepsilon =\frac{\varepsilon _0}{2^{j+1}}\min(1,d(\{w_1,\ldots,w_q\},K_{\psi(j+1)}),d(\{w_1,\ldots,w_q\},L))$,
%\item $a_i=-f_{j}(w_i)$ for $1\leq i \leq q$,
%\item $K=K_{\psi(j+1)}$, $L=L_{j+1}$,
%\item $h=Q\cdot Q_{\phi(j+1)}-f_j$,
%\item $p=\min\left\{p^{\psi(j+1)}_n;\,p^{\psi(j+1)}_n\geq \text{deg}(f_{j})\right\}+q+1$,
%\end{itemize}
%where we use the same notations as in the proof of Theorem \ref{main-thm}). The end of the proof is standard.
\end{proof}

%Theorem \ref{main-thm-Q} will have several corollaries that we present in the next subsection.

\subsection{Algebraic genericity of a set of Pad\'e universal series}

Pad\'e approximation is not a linear operation and thus algebraic genericity for Pad\'e universal series is not a straightforward property. However, using Theorem \ref{main-thm-Q} with a particular polynomial $Q$, one can show that a set of Pad\'e universal series, related to $\UU_{Q}^{\mu}(\D)$, contains a linear subspace (except for the zero function) dense in $H(\D)$.
%The goal of this section is the following theorem.
\begin{theorem}\label{alg-gene}Let $\mu$ be an unbounded sequence of positive integers, $a\in \C \setminus \D$, and $q\geq 1$ an integer. The set 
$$\tilde\UU^{\mu}_{Q_{q}}(\D):=\bigcup_{p=0}^{q}\UU_{Q_{p}}^{\mu}(\D),
\qquad Q_{p}(z)=(1-z/a)^{p},$$
is algebraically generic. More precisely, there exists a sequence of functions $f_{k}\in
\UU_{Q_{q}}^{\mu}(\D)$, $k\geq 0$, such that $F=\text{span}(f_{k},k\geq 0)$ is dense in $H(\D)$ and $F\setminus\{0\}$ is contained in $\tilde\UU^{\mu}_{Q_{q}}(\D)$.
% i.e. it
%contains, apart from the zero function, a dense linear subspace of $H(\D)$.
\end{theorem}

\begin{proof} The scheme of proof is standard, see e.g.  \cite[Corollary 19]{bgnp}. Let $\left(O_j\right)_{j\geq 0}$ be a countable basis of neighboorhoods of the separable space $H(\D)$ and let $\left(K_i\right)_i$ be a sequence of compact subsets given by Lemma \ref{lemma-geo}.
%with $w_1=\ldots=w_q=a$. 
%Let $Q=(1-z/a)^q$. Observe that the zero of $Q$ is outside any $K_n$. 
We set $\mu ^i_0=\mu = \left(p_n\right)_n$ for any $i\in\N$. By Corollary \ref{last-coro-gene} there exists $f_0 \in \bigcap _{i\in \N}\UU_{Q_{q}}^{\mu_0^i}(\D,K_{i})\cap O_0$. For any $i\in \N$, let $\mu _1^i :=\left(p_n^{1,i}\right)_n$ be a subsequence of $\mu^i_0$ such that
\begin{enumerate}
\item $[f_0;p_n^{1,i}/q]\rightarrow 0$ as $n\rightarrow \infty$ in $A(K_{i})$;% uniformly on $K_i$;
\item $[f_0;p_n^{1,i}/q]\rightarrow f_0$ as $n\rightarrow \infty$ in $H(\D)$;
\item $Q_{q}$ is the denominator of $[f_0;p_{n}^{1,i}/q]$ for every $n\geq 0$.
\end{enumerate}
We then define $f_1\in \bigcap _{i\in \N}\UU_{Q_{q}}^{\mu^i _1}(\D,K_i)\cap O_1$ and construct by induction sequences $\left(f_k,\mu ^i_k\right)_k$ with $\mu _k^i=\left(p_n^{k,i}\right)_n$, $i \in \N$, such that for any $i$,
\begin{enumerate}
\item [4.]$f_k\in \bigcap _{i\in \N}\UU_{Q_{q}}^{\mu^i _k}(\D,K_i)\cap O_k$;
\item [5.] $\mu _{k+1}^i$ is a subsequence of $\mu _k^i$, for any $i$;
\item [6.] $[f_k; p_n^{j,i}/q] \rightarrow 0$ as $n\rightarrow \infty$ in $A(K_i)$ for any $j> k$;
\item [7.] $[f_k; p_n^{j,i}/q] \rightarrow f_k$ as $n\rightarrow \infty$ in $H(\D)$ for any $j> k$;
\item [8.] $Q_{q}$ is the denominator of $[f_k; p_n^{j,i}/q]$ for any $n\geq 0$ and any $j> k$.
\end{enumerate}
We assert that $F=\text{span}\left(f_k,\,k\geq 0\right)$ satisfies the properties of the theorem. It is clearly dense in $H(\D)$ since the family $\left(f_k\right)_k$ is dense in $H(\D)$. Let now
$$g=\alpha_0 f_0+\ldots \alpha _l f_l,\qquad \alpha _l\neq 0,$$
be any nonzero function of $F$.
We have to show that $g\in \tilde\UU^{\mu,}_{Q_{q}}(\D)$. Let $K\subset \C \setminus (\D\cup \{a\})$ and $h\in A(K)$. Choose $i_0$ such that $K\subset K_{i_0}$. Since $f_l \in \bigcap _{i\in \N}\UU_{Q_{q}}^{\mu ^i_l}(\D,K_i)$ there exists a subsequence $\left(\lambda _n\right)_n$ of $\mu ^{i_0}_l$ such that $Q$ is the denominator of $[f_k;p_{\lambda _n}^{l,i_0}/q]$ for any $0\leq k\leq l$ and such that
$$[f_l;p_{\lambda _n}^{l,i_0}/q]\to\frac{h}{\alpha _l}\text{ in }A(K)\quad\text{and}\quad[f_l;p_{\lambda _n}^{l,i_0}/q]\rightarrow f_l\text{ in }H(\D),\text{ as }n\rightarrow \infty,$$
while, for every $0\leq k<l$, 
$$[f_k; p_{\lambda _n}^{l,i_0}/q] \rightarrow 0 \text{ in } A(K)\quad\text{and}\quad[f_k; p_{\lambda _n}^{l,i_0}/q] \rightarrow f_k\text{ in }H(\D),\text{ as }n\rightarrow \infty.$$
Moreover from items 3. and 8., the Pad\'e approximant $[g,p_{\lambda _n}^{l,i_0}/q]$ exists and, for every $n\geq 0$
\begin{equation}\label{pade-g}
[g;p_{\lambda _n}^{l,i_0}/q]=\sum _{k=0}^l\alpha_{k}[f_k;p_{\lambda _n}^{l,i_0}/q].
\end{equation}
Thus,
$$ [g;p_{\lambda _n}^{l,i_0}/q]\to h\text{ in }A(K)\quad\text{and}\quad[g;p_{\lambda _n}^{l,i_0}/q]\rightarrow g\text{ in }H(\D),\text{ as }n\rightarrow \infty.$$
There exists at least a degree $0\leq p_{0}\leq q$ such that the denominator of (\ref{pade-g}), written in irreducible form, equals $Q_{p_{0}}$ for infinitely many $n$. Consequently, $g\in\UU_{Q_{p_{0}}}^{\mu}(\D)\subset \tilde\UU_{Q_{q}}^{\mu}(\D)$.
\end{proof}
\subsection{Spaceability of a set of Pad\'e universal series}
In this section, we prove that the set $\tilde\UU_{Q_{q}}^{\mu}(\D)$ is also spaceable.

\begin{theorem}\label{spaceability}The set $\tilde\UU_{Q_{q}}^{\mu}(\D)$ is spaceable, namely, it contains, except for the zero function, an infinite dimensional closed linear subspace of $H(\D)$. More precisely, with the notations of Theorem \ref{alg-gene}, there exists a sequence of functions $f_{k}\in
\UU_{Q_{q}}^{\mu}(\D)$, $k\geq 0$, such that $F=\text{span}(f_{k},k\geq 0)$ is a closed infinite dimensional subspace of $H(\D)$ and $F\setminus\{0\}$ is contained in $\tilde\UU^{\mu}_{Q_{q}}(\D)$.
\end{theorem}

The first construction of an infinite dimensional closed subspace of universal Taylor series was given in \cite{Bay}. Bayart's proof has been extended to an abstract setting in Banach spaces in \cite{CH} and then in Fr\'echet spaces in \cite{Men}. Our proof will consist in a slight adaptation of the method presented in these papers.
It is now standard and uses different ingredients, the first one being
the notion of a basic sequence, see e.g. \cite{Men,Pet}. A sequence $\left(u_n\right)_n$ in a Fr\'echet space $X$ (over a field $\mathbb{K}$) is a \emph{basic sequence} if it is a Schauder basis of $\text{clos}(\text{span}(\{u_n,\,n\geq 0\}))$, i.e. if any element $x\in \text{clos}(\text{span}(\{u_n,\,n\geq 0\}))$ can be uniquely written in $X$ as a series $\sum_n a_nu_n$ for some $a_n\in \mathbb{K}$, $n\in \N$. If $X$ is a Fr\'echet space with a continuous norm and if $\left(p_n\right)_n$ is a sequence of increasing continuous norms defining its topology (the existence of which is equivalent to the existence of a continuous norm in $X$), we shall denote by $X_n$ the normed space $X$ endowed with the topology of the norm $p_n$.

The next result gives a criterion to check that a sequence of elements in a Fr\'echet space with a continuous norm is basic.

\begin{lemma}[{\cite[Lemme 2.2]{Men}}]%\label{crit-basic-sequence}
Let $X$ be a Fr\'echet space (over the field $\mathbb{K}=\R \text{ or } \C$) with a continuous norm, $\left(p_n\right)_n$ an increasing sequence of continuous norms defining its topology and $\left(\varepsilon _n\right)_n$ a sequence of positive real numbers such that $B=\prod _n (1+\varepsilon _n)<\infty$. If $\left(u_n\right)_n$ is a sequence of elements in $X$ satisfying for every $n\in \N$, every $0\leq j\leq n$, and every $a_1,\ldots , a_{n+1} \in \mathbb{K}$
$$p_j\left(\sum_{k=0}^{n}a_ku_k\right)\leq (1+\varepsilon _n)p_j\left(\sum_{k=0}^{n+1}a_ku_k\right)$$
then $\left(u_n\right)_n$ is a basic sequence in each $X_n$, $n\in \N$, and in $X$.
\end{lemma}

The infimum of the constant $B=\prod _n (1+\varepsilon _n)$ satisfying the above is called the constant of basicity of $\left(u_n\right)_n$. A useful tool to construct convenient basic sequences in Fr\'echet spaces with continuous norm is as follows.

\begin{lemma}[{\cite[Lemme 2.3]{Men}}]\label{basic-sequence}Let $X$ be a Fr\'echet space (over the field $\mathbb{K}=\R \text{ or } \C$) with a continuous norm, $\left(p_n\right)_n$ an increasing sequence of continuous norms defining its topology and $M$ an infinite dimensional subspace of $X$. Then for every $\varepsilon >0$, every $u_0,\ldots,u_n \in X$, there exists $u_{n+1}\in M$ such that $p_1\left(u_{n+1}\right)=1$ and such that for every $0\leq j\leq n$, for every $a_0,\ldots,a_{n+1} \in \mathbb{K}$ we have
$$p_j\left(\sum_{k=0}^{n}a_ku_k\right)\leq (1+\varepsilon)p_j\left(\sum_{k=0}^{n+1}a_ku_k\right).$$
\end{lemma}

We shall also need the notion of equivalent basic sequences.

\begin{definition}Two basic sequences $\left(u_n\right)_n$ and $\left(f_n\right)_n$ of a Fr\'echet space $X$ are \emph{equivalent} if, for any $a_1,a_2,\ldots \in \mathbb{K}$, $\sum _n a_nu_n$ converges in $X$ if and only if $\sum _n a_nf_n$ converges in $X$.
\end{definition}

The following lemma is also useful.

\begin{lemma}[{\cite[Lemme 2.5]{Men}}]\label{lem-basic}
Let $X$ be a Fr\'echet space (over the field $\mathbb{K}=\R \text{ or } \C$) with a continuous norm, $\left(p_n\right)_n$ an increasing sequence of continuous norms defining its topology and let $K\geq 1$. If $\left(u_k\right)_k$ is a basic sequence in $X$ such that for every $k$, $p_1(u_k)=1$, and for every $n\in \N$, the sequence $\left(u_k\right)_{k\geq n}$ is basic in $X_n$ with constant of basicity less than $B$, then every sequence $\left(f_k\right)_k$ in $X$ satisfying
$$\sum _n 2Bp_n\left(u_n-f_n\right)<1$$
is basic in each $X_n$, $n\in \N$, and in $X$.
Moreover $\left(u_k\right)_k$ and $\left(f_k\right)_k$ are equivalent in the completion of each $X_n$, $n\in \N$, and in $X$.
\end{lemma}

Theorem \ref{spaceability} is stated in the space $H(\D)$, where $(z^{n})_{n\geq 0}$ can be taken as an explicit basic sequence. Hence the above lemmas on basic sequences are not crucial in this case. However, in the more general space $H(\Omega)$, $\Omega$ a simply connected domain in $\C$, the existence of a basic sequence is no more obvious, and these lemmas would be mandatory there.
\begin{proof}[Proof of Theorem \ref{spaceability}]Let \\
-- $\left(L_n\right)_n$ be an exhaustion of compact subsets of $\D$,\\
-- $\left(K_n\right)_n \subset \C \setminus (\D \cup \{a\})$, a family of compact sets given by Lemma \ref{lemma-geo},\\
-- $\left(P_n\right)_n$ an enumeration of polynomials with coefficients in $\Q+i\Q$,\\
-- $\phi,\psi$, maps from $\N$ to $\N$ such that for any pair $(l,r)\in \N\times \N$, there exist infinitely many integers $n$ with $(\phi(n),\psi(n))=(l,r)$. \\
Moreover, we denote by $\prec$ the (non strict) lexicographical order on $\N\times \N$.

Three sequences $\left(u_{k}\right)_{k\geq 0}$, $\left(g_{n,k}\right)_{n\geq k\geq 0}$ and $\left(f_{n,k}\right)_{n\geq k\geq 0}$ of polynomials in $H(\D)$ can be built inductively, from which the desired infinite dimensional closed linear subspace of $H(\D)$ is derived. The construction of these sequences is described in all details in \cite{Men}, see also \cite{CH}. Figure \ref{diagr} illustrates the ordering of the construction, which follows the lexicographical order of the indices. The construction 
would make use of Lemma \ref{basic-sequence} and Lemma \ref{mergelyan}, as indicated in the figure.
%follows the lexicographical order as in \cite[Theorem ?]{Men}, see Figure \ref{diagr}. in which paper all the details are given (see also \cite{CH}). It is illustrated by the following diagram in which the numbers inside boxes corresponds to the steps.
\begin{figure}[htb]
\centering
\begin{tikzpicture}[->,thick,>=stealth']

  \tikzstyle{state} = [draw, very thick, fill=white, rectangle, minimum height=3em, minimum width=5em]
  \tikzstyle{stateEdgePortion} = [black,thick];

% double arrow style
\tikzstyle{doublearrow}=[draw, color=black!75, draw=black!50, thick, double distance=2pt, ->, >=stealth]

  %
  % Position States
  %
  %1ere ligne
  \node[state, name=A11, text width=2cm] 
  { $\hphantom{abcd}u_{0}$ $\,\boxed{1}$\\$\hphantom{a}\swarrow$ \\$g_{0,0}\to f_{0,0}$};
  %2ieme ligne
\node[state, name=A21, below of=A11, node distance=2.5cm,text width=1.9cm] 
  {$\hphantom{abcdefg}\boxed{2}$\\$g_{1,0}\to f_{1,0}$};
\node[state, name=A22, right of=A21, node distance=4cm, text width=1.9cm] 
{$\hphantom{abcd}u_{1}$ $\,\boxed{3}$\\$\hphantom{a}\swarrow$\\$g_{1,1}\to f_{1,1}$};
%3ieme ligne
\node[state, name=A31, below of=A21, node distance=2.5cm,text width=1.9cm] 
  {$\hphantom{abcdefg}\boxed{4}$\\$g_{2,0}\to f_{2,0}$};
\node[state, name=A32, below of=A22, node distance=2.5cm,text width=1.9cm] 
  {$\hphantom{abcdefg}\boxed{5}$\\$g_{2,1}\to f_{2,1}$};  
\node[state, name=A33, right of=A32, node distance=4cm, text width=1.9cm] 
{$\hphantom{abcd}u_{2}$ $\,\boxed{6}$\\$\hphantom{a}\swarrow$\\$g_{2,2}\to f_{2,2}$};
%4ieme ligne
\node[state, name=A41, below of=A31, node distance=2.5cm,text width=1.9cm] 
  {$\hphantom{abcdefg}\boxed{7}$\\$g_{3,0}\to f_{3,0}$};
\node[state, name=A42, below of=A32, node distance=2.5cm,text width=1.9cm] 
  {$\hphantom{abcdefg}\boxed{8}$\\$g_{3,1}\to f_{3,1}$};  
\node[state, name=A43, below of=A33, node distance=2.5cm,text width=1.9cm] 
  {$\hphantom{abcdefg}\boxed{9}$\\$g_{3,2}\to f_{3,2}$};
\node[state, name=A44, right of=A43, node distance=4cm, text width=2.05cm] 
{$\hphantom{abcd}u_{3}$ $\,\boxed{10}$\\$\hphantom{a}\swarrow$\\$g_{3,3}\to f_{3,3}$};
%
% Path
%
\path (A11.25) edge[thick, out=0, in=90] node[auto]{Lemma \ref{basic-sequence}} (A22.80);
\path (A11) edge[doublearrow] node[auto]{Lemma \ref{mergelyan}}  (A21);
%\path (A21.east) edge[thick, bend left=70, out=70] node[pos=0.6, above]{Lemma 3.14} (A22.north);
\path (A21) edge[doublearrow] node[auto]{Lemma \ref{mergelyan}}  (A31);
%\path (A22) edge[thick] (A31.60);
\path (A22.25) edge[thick, out=0, in=90] node[auto]{Lemma \ref{basic-sequence}} (A33.80);
\path (A22) edge[doublearrow] node[auto]{Lemma \ref{mergelyan}}  (A32);
%\path (A31.east) edge[thick, bend left=70, out=70] node[pos=0.5, below]{Lemma 3.14} (A32);
%\path (A32.east) edge[thick, bend left=70, out=70] node[pos=0.6, above]{Lemma 3.14} (A33.north);
\path (A33.25) edge[thick, out=0, in=90] node[auto]{Lemma \ref{basic-sequence}} (A44.80);
\path (A31) edge[doublearrow] node[auto]{Lemma \ref{mergelyan}}  (A41);
\path (A32) edge[doublearrow] node[auto]{Lemma \ref{mergelyan}}  (A42);
\path (A33) edge[doublearrow] node[auto]{Lemma \ref{mergelyan}}  (A43);
  \end{tikzpicture}
  \caption{The first steps in the construction of the sequences $(u_{k})_{k\geq 0}$, $(g_{n,k})_{n\geq k\geq 0}$ and $(f_{n,k})_{n\geq k\geq 0}$. Numbers in boxes correspond to the numbering of these steps.}
\label{diagr}
\end{figure}
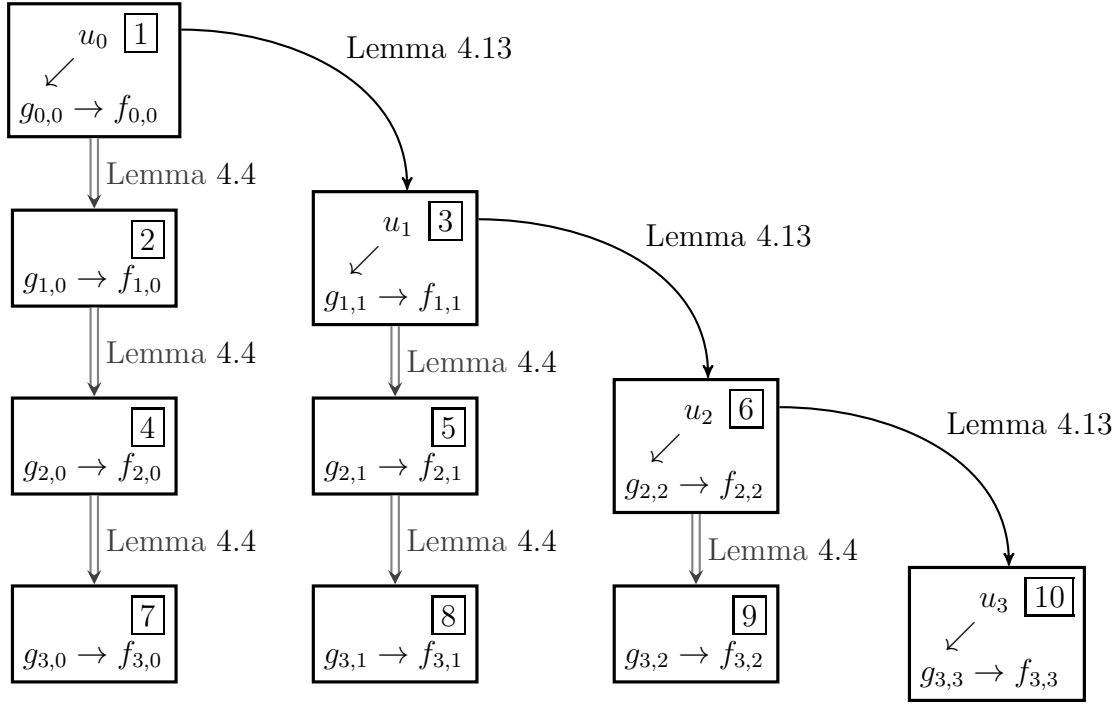

Here, we only state the properties of the 
% similarly as in the proof of Theorem \ref{main-thm-Q}, it makes that the 
resulting sequences. For $\left(\eta _n\right)_n\subset \R_+$, a well-chosen sequence decreasing to $0$ fast enough and for $\left(u_k\right)_k$, a basic sequence of $H(\D)$, obtained from Lemma \ref{basic-sequence} (in the disk $\D$, $(u_{k})_{k}$ could be a subsequence of $(z^{n})_{n}$), the sequences $\left(u_{k}\right)_{k\geq 0}$, $\left(g_{n,k}\right)_{n\geq k\geq 0}$ and $\left(f_{n,k}\right)_{n\geq k\geq 0}$ satisfy, for every $n\geq k\geq 0$:
%\smallskip{}
\begin{enumerate}
\item $\displaystyle{\left\Vert Q_{q} P_{\phi(n)}-g_{n,k}\right\Vert _{K_{\psi(n)}}\leq \eta_n\min(1,d(a,K_{\psi(n)})^q)}$;
\item $\displaystyle{\left\Vert f_{n,k}\right\Vert _{K_{\psi(n+1)}} \leq \eta_n d(a,K_{\psi(n+1)})^q}$;
\item $\displaystyle{\left\Vert f_{n+1,k}-f_{n,k}\right\Vert _{L_{n+1}}\leq \eta_n d(a,L_{n+1})^q}$;
\item $\displaystyle{\left\Vert f_{n,k}-g_{n,k}\right\Vert _{L_{n}} \leq \eta_n d(a,L_{n})^q}$;
\item $\displaystyle{\left\Vert f_{k,k}-Q_{q} u_k\right\Vert _{L_{k}} \leq \eta_k d(a,L_{k})^q}$;
\item For every $n\geq k$, $\displaystyle{g_{n+1,k}=f_{n,k}+P}$ (resp. $\displaystyle{g_{k,k}=u_k+P}$) where $P$ is a polynomial with
\begin{eqnarray*}\text{val}(P) & \geq & \min \left\{p_n;\,p_n\geq \max _{(n',k')\prec(n,k)}\text{deg}\left(f_{n',k'}\right)+q+1\right\}\\
( \text{resp.} & \geq & \min \left\{p_n;\,p_n\geq \text{deg}\left(u_k\right)+q+1\right\});
\end{eqnarray*}
\item For every $n\geq k$, $\displaystyle{f_{n,k}=g_{n,k}+R}$ where $R$ is a polynomial with
$$\text{val}(R)\geq \min \left\{p_n;\,p_n\geq \max _{(n',k')\prec(n,k)}\text{deg}\left(g_{n',k'}\right)+q+1\right\};$$
%\item $\displaystyle{\text{val}(g_{n,k})\geq \min \left\{p_n;\,p_n\geq \max _{(n',k')\prec(n,k)}\text{deg}\left(f_{n',k'}\right)+q+1\right\}$;
%\item $\displaystyle{\text{val}(f_{n,k})\geq \min \left\{p_n;\,p_n\geq \max _{(n',k')\prec(n,k)}\text{deg}\left(g_{n',k'}\right)+q+1\right\}$;
\item $g_{n,k}$ does not vanish at $a$;
\item $\displaystyle{\left\Vert u_k\right\Vert _{L_{0}} =1}$.
\end{enumerate}
%\smallskip{}
For every $k\geq 0$, we set 
$$\tilde{f_k}=\sum _{n\geq k}\left(f_{n+1,k}-f_{n,k}\right)+f_{k,k}\quad\text{and}\quad f_k=\tilde{f_k}/Q_{q}.$$
By item 3., we deduce that $f_k\in H(\D)$ and that $f_k=\lim _{n\rightarrow \infty} f_{n,k}/Q_{q}$. Moreover by item 5. and Lemma \ref{lem-basic}, $\left(f_k\right)_k$ is a basic sequence of $H(\D)$ equivalent to $\left(u_k\right)_k$. Thus we define the closed subspace
$$F:=\text{clos}(\text{span}(f_k,\,k\geq 0))=\Biggl\{\sum _{k=0}^{\infty}a_kf_k
%,\,\sum _{k=0}^{\infty}a_kf_k 
\text{ which converges in } H(\D)\Biggr\}.$$
Since, by items 6. and 7., the $f_k$'s are linearly independent, $F$ is infinite dimensional. It remains to check that every nonzero element $f$ of $F$ is in $\tilde\UU^{\mu}_{Q_{q}}(\D)$. Fix any $j> k\geq 0$. First, arguing as in the proof of Theorem \ref{main-thm}, observe that, in view of items 6., 7. and 8., if we denote by $p_{n_j}(k)$ the smallest element of $\mu$ such that $p_{n_j}(k)\geq \text{deg}\left(g_{j,k}\right)$, we have
$$f_k\in \DD_{p_{n_j}(k),q} \text{ and } \displaystyle{[f_k;p_{n_j}(k)/q]=\frac{g_{j,k}}{Q_{q}}}.$$
Let $f=\sum _{k=0}^{\infty}a_kf_k \in F\setminus\{0\}$. From the construction, in particular the choices of valuations, we obtain that the $p_{n_j}(k)+q$ Taylor coefficients of $f$ coincide with those of
$$a_k\frac{g_{j,k}}{Q_{q}}+\sum _{k'=0}^{k-1}a_{k'}\frac{f_{j,k'}}{Q_{q}}+
\sum_{k'=k+1}^{j-1}a_{k'}\frac{f_{j-1,k'}}{Q_{q}},$$
which is a rational function whose numerator and denominator degrees are less than or equal to $p_{n_j}(k)$ and $q$ respectively. Thus the Pad\'e approximant of $f$ of type $\left(p_{n_j}(k),q\right)$ exists and equals
$$[f;p_{n_j}(k)/q]=a_k\frac{g_{j,k}}{Q_{q}}+\sum _{k'=0}^{k-1}a_{k'}\frac{f_{j,k'}}{Q_{q}}+\sum _{k'=k+1}^{j-1}a_{k'}\frac{f_{j-1,k'}}{Q_{q}}.$$
Note that the denominator of the rational fraction in the right-hand side, written in irreducible form, equals $Q_{p}$ for some $0\leq p\leq q$.

To finish we have to show that $f$ possesses universal approximation properties. 
%in the sense of Definition \ref{defi-new-class}. 
We first observe that the sequence $\left(a_k\right)_k$ is bounded by some constant $M$. Indeed since $\left(f_k\right)_k$ and $\left(u_k\right)_k$ are equivalent, the series $\sum _{k\geq 0}a_ku_k$ converges and by item 9., we have
$$\left\vert a_k\right\vert= \left\Vert a_ku_k\right\Vert _{L_0}\leq 2B\bigg\| \sum _{k\geq 0}a_ku_k\bigg\|_{L_0},$$
where $B$ stands for the constant of basicity of the basic sequence $\left(u_k\right)_k$. We set $M=2B$. Let $k_0$ be the smallest integer $k$ such that $a_k\neq 0$. Now let $K\subset \C \setminus (\D \cup \{a\})$ be a compact subset and let $r\in \N$ be such that $K\subset K_r$. We fix also a polynomial $P_l$ in the countable dense family of polynomials. By definition of $\phi$ and $\psi$ there exists an increasing sequence $\left(v_j\right)_j\subset \N$ such that $\left(\phi(v_j),\psi(v_j)\right)=(l,r)$ and $v_j > k_0$ for any $j\geq 0$. Let us denote by $p_{n_j}$ the smallest element of $\mu$ such that $p_{n_j}\geq \text{deg}\left(g_{v_j,k_0}\right)$. By the above and items 1. and 2., we have
\begin{eqnarray*}\left\Vert [f;p_{n_j}/q]-P_l\right\Vert _{K_r}& = & \left\Vert a_{k_0}\frac{g_{v_j,k_0}}{Q_{q}}+\sum _{k'=k_0+1}^{v_j-1}a_{k'}\frac{f_{v_j-1,k'}}{Q_{q}}-P_{\phi(v_j)}\right\Vert _{K_{\psi(v_j)}}\\
& \leq & \left\Vert a_{k_0}\frac{g_{v_j,k_0}}{Q_{q}}-P_{\phi(v_j)}\right\Vert _{K_{\psi(v_j)}} + \left\Vert\sum _{k'=k_0+1}^{v_j-1}a_{k'}\frac{f_{v_j-1,k'}}{Q_{q}}\right\Vert _{K_{\psi(v_j)}} \\
& \leq & M\eta_{v_j} + M \sum _{k'=k_0+1}^{v_j-1}\eta_{v_j}.
\end{eqnarray*}
Choosing $\left(\eta _n\right)_n$ decreasing to $0$ fast enough, we get that $[f;p_{n_j}/q]$ tends to $P_l$ in $A(K_{r})$ as $j$ tends to $\infty$.

It remains to show that, given $m\in \N$, $[f;p_{n_j}/q]$ tends to $f$ uniformly on $L_m$. For $j$ large enough, using items 3. and 4. (see also Figure \ref{diagr}), one can check that
\begin{eqnarray*}\left\| [f;p_{n_j}/q]-f\right\|_{L_n}& \leq & \Biggl\| \sum _{k'\geq v_j+1}a_{k'}f_{k'}\Biggr\|_{L_n} 
+ \Biggl\|\sum _{k'=k_0}^{v_j}a_{k'}\Bigl(\sum _{n\geq v_j+1}\frac{f_{n+1,k'}}{Q_{q}}-\frac{f_{n,k'}}{Q_{q}}\Bigr)\Biggr\|_{L_n} \\
&  & + \left\Vert a_{v_j}\left(f_{v_j,v_j}-g_{v_j,v_j}\right)\right\Vert _{L_n}\\
& \leq & \Biggl\|\sum _{k'\geq v_j+1}a_{k'}f_{k'}\Biggr\|_{L_n} + \sum _{k'=k_0}^{v_j}|a_{k'}|\Bigl(\sum _{n\geq v_j+1}\eta _n\Bigr) + \left\vert a_{v_j}\right\vert \eta _{v_j}.
\end{eqnarray*}
Again, choosing $\left(\eta _n\right)_n$ to decrease to $0$ fast enough gives our contention since $\sum _{k\geq 0}a_{k}f_{k}$ is convergent in $H(\D)$.
\end{proof}

%\begin{remark}When reading the previous proof, the introduction of the material about basic sequences may seem a bit superfluous since the sequence $\left(z^n\right)_n$ is a trivial basic sequence in $H(\D)$. Yet as announced at the beginning of Section 3, our results hold for $\D$ replaced with any simply connected domain $\Omega$, and to construct basic sequences in $H(\Omega)$ is not trivial any more. So this justifies this introduction.
%\end{remark}

%We immediately deduce from Theorems \ref{alg-gene} and \ref{spaceability} the following corollary.

%\begin{corollary}Let $\mu$ be a strictly increasing sequence of integers, $q$ an integer and $K\subset \C \setminus \D$ a compact set. The set $\UU^{[\mu,q]}(\D,K)$ is residual, algebraically generic and spaceable in $H(\D)$.
%\end{corollary}

\subsection{Genericity of Pad\'e universal series whose Pad\'e approximants have asymptotically prescribed poles}

We have seen in Theorem \ref{main-thm-Q} that the set $\UU_{Q}^{\mu}(\D)$ is meagre in $H(\D)$. However, if we only demand the Pad\'e approximants to have poles that go \emph{asymptotically} to prescribed points in $\C$, then the corresponding set of Pad\'e universal series becomes residual. 
In the sequel, we adopt the following convention.
%For $f\in \DD_{p,q}$, we denote by $(a_{p,q}^{(1)}(f),\ldots,a_{p,q}^{(q)}(f))$ the poles of $[f;p/q]$ 
When we write a tuple $(w^{(1)},\ldots ,w^{(q)})$ of complex points, we assume it is
ordered with respect to some lexicographical order $\prec _{\alpha}$ on $\R_+\times [\alpha,\alpha+2\pi[$, where each point is identified with its polar coordinates $(r,\vartheta)\in \R_+\times [\alpha,\alpha+2\pi[$. The argument $\alpha\in[0,2\pi[$ is chosen to be different from each of the arguments of the $w^{(i)}$, $1\leq i\leq q$.
%Here we adopt the following convention: 
%When $\text{deg}(Q_{p,q}(f))=r<q$, we agree to write $a_{p,q}^{(s)}=\infty$ for every $r<s\leq q$.

\begin{theorem}\label{2nd-thm-q}Let $\mu$ be an unbounded sequence of integers and $q\geq 1$ an integer. There exists a function $f$ in $H(\D)$ such that for every tuple $W=(w^{(1)},\ldots ,w^{(q)})$ of points in $\C\setminus \D$ ordered as above, where the argument $\alpha$ is different from the argument of $w^{(i)}$, $1\leq i\leq q$, for every compact set $K\subset \C \setminus (\D\cup W)$ with $\C \setminus K$ connected, and every function $h\in A(K)$, there is a subsequence $\left(\lambda_n\right)_n$ of $\mu$ with the following properties:
\begin{enumerate}[a)]
\item $f \in \NN_{\lambda_n,q}$,
%and 
%every $a_{\lambda_n,q}^{(i)}(f)$, $1\leq i\leq q$, is finite;
%$\deg Q_{\lambda_{n},q}(f)=q$, 
for all $n\geq 1$;
\item $[f;\lambda_n/q]\rightarrow h$ in $A(K)$ as $n\rightarrow \infty$;
\item $[f;\lambda_n/q] \rightarrow f$ in $H(\D)$ as $n\rightarrow \infty$;
\item $a_{\lambda_{n}}^{(i)}\rightarrow w^{(i)}$ as $n\rightarrow \infty$, $1\leq i \leq q$, where $(a_{\lambda_{n}}^{(1)},\ldots,a_{\lambda_{n}}^{(q)})$ are the roots of $Q_{\lambda_{n},q}(f)$, ordered with respect to $\prec_{\alpha}$.
% where $\alpha\in [0,2\pi[$ is different from $\arg{w^{(i)}}$ for every $1\leq i\leq q$.
\end{enumerate}
The set $\tilde{\UU}^{\mu,q}(\D)$ of such functions is a dense $G_{\delta}$-subset of $H(\D)$.
\end{theorem}

\begin{proof}
The main tool is Baire category theorem. Let \\
-- $(w^{(1)}_i,\ldots,w^{(q)}_{i})_i$ be an enumeration of all ordered tuple of $q$ complex numbers in $\C\setminus \D$ with rational coordinates,\\
-- $(L_r)_r$, an exhaustion of compacta in $\D$,\\
-- $\left(K_{n,i}\right)_n$, the sequence of compacta given by Lemma \ref{lemma-geo-Q} with $(w_{1},\ldots,w_{q})$ replaced by $(w_i^{(1)},\ldots,w_{i}^{(q)})$,\\
-- $\left(P_j\right)_j$ an enumeration of polynomials with coefficients in $\Q+i\Q$.\\
Applying Mergelyan's theorem in each $K_{n,i}$,  the set $\tilde{\UU}^{\mu,q}(\D)$ can be described as:
$$\tilde{\UU}^{\mu,q}(\D)=\bigcap_{i,n,j,s,r}\,\bigcup_{p\in \mu}F(i,n,j,s,r,p)$$
where, for $i,n,j,r \in \N$, $s\in \N^*$ and $p\in \mu$, $F(i,n,j,s,r,p)$ denotes the set of functions in $H(\D)$ such that
%\begin{multline*}
$$f\in \DD_{p,q},\quad\deg Q_{\lambda_{n},q}(f)=q,\quad\max_{1\leq l\leq q}|a_{p}^{(l)}-w_i^{(l)}|<s^{-1},$$
and
$$
\Vert [f;p/q]-P_j\Vert_{K_{n,i}}<s^{-1},\quad\Vert [f;p/q]-f\Vert_{L_r}<s^{-1}.
$$
%Let $\left((w^{(l)}_i)_{1\leq l\leq q}\right)_i$ be an enumeration of all families of $q$ complex numbers in $\C\setminus \D$ with rational coordinates. Let $(L_r)_r$ be an exhaustion of compacta in $\D$. For $i\in \N$, if $\left(K_{n,i}\right)_n$ is the sequence of compacta given by Lemma \ref{lemma-geo-Q} with $(w_l)_{1\leq l\leq q}$ replaced with $(w^{(l)}_i)_{1\leq l\leq q}$, then $\tilde{\UU}^{\mu,q}(\D)$ can be described as follows:
%$$\tilde{\UU}^{\mu,q}(\D)=\bigcap_{i,n,j,s,r}\,\bigcup_{p\in \mu}F(i,n,j,s,r,p)$$
%where, for $i,n,j,r \in \N$, $s\in \N^*$ and $p\in \mu$,
%\begin{multline*}F(i,n,j,s,r,p)=\big\{f\in H(\D);\,f\in \DD_{p,q},\max _{1\leq l\leq q}
%|a_{p,q}^{(l)}(f)-w_i^{(l)}|<s^{-1},\\\Vert [f;p/q]-P_j\Vert_{K_{n,i}}<s^{-1},\,\Vert [f;p/q]-f\Vert_{L_r}<s^{-1}\big\}.
%\end{multline*}
Recall that, by definition, $f\in\DD_{p,q}$ means that the Hankel determinant $C_{p,q}$ does not vanish. Also, from the determinantal expression for the denominator of a Pad\'e denominator, see (\ref{deter}), we have that $Q_{\lambda_{n},q}(f)$ is of degree $q$ if and only if $C_{p+1,q}$ is nonzero. Now, by Baire category theorem, it is sufficient to prove that each union $\bigcup_{p\in \mu}F(i,n,j,s,r,p)$ is open and dense in $H(\D)$. The Hankel determinants $C_{p,q}$, $C_{p+1,q}$, the Pad\'e approximant $[f;p/q]$ and the roots $a_{p}^{(l)}$, $1\leq l\leq q$, are continuous functions of the $p+q+1$ first coefficients of the Taylor expansion of $f$. Therefore each set $F(i,n,j,s,r,p)$ is open. Now Theorem \ref{main-thm-Q}, with $(w_1,\ldots ,w_q)$ replaced by $(w_i^{(1)},\ldots,w_i^{(q)})$, ensures that each $\cup_{p\in \mu}F(i,n,j,s,r,p)$ is dense in $H(\D)$.
\end{proof}

\begin{remark} Theorem \ref{2nd-thm-q} gives another occurence of density of poles. Indeed, the functions in $\tilde{\UU}^{\mu,q}(\D)$ have Pad\'e approximants of degree $(p_n,q)$, with $\left(p_n\right)_n=\mu$, whose set of poles is dense in $\C \setminus \D$, compare with the results of Section \ref{beh-poles}.
\end{remark}

\section{From universal series to Pad\'e universal series}\label{TtoP}

In this section we explore connections between the classical universal series and the Pad\'e universal series. 
%There is a vast literature about universal series, see \cite{bgnp,Luh2,Nes} and the references therein. For completeness, we recall the notion of a universal series.
We recall that the notion of universal series was defined in the introduction.
%\begin{definition}\label{defi-univ-Taylor-series} Let $\Omega \subset \C$ be a simply connected domain containing $0$ and $\mu$ an increasing sequence of positive integers. A universal series is a function $f$ in $H(\Omega)$ such that, for every compact set $K\subset\C \setminus \D$ 
%with connected complement, and every function $h\in A(K)$ there exists 
%a subsequence $(\lambda_n)_n$ of $\mu$ with the following properties:
%\begin{enumerate}[a)]
%\item The partial sums $S_{\lambda _{n}}\left(f\right)$ of the Taylor expansion of $f$ at $0$ converges to $h$ uniformly on $K$ as $n$ tends to $\infty$;
%\item $S_{\lambda _{n}}\left(f\right)$ converges to $f$ uniformly on every compact subset of $\Omega$ as $n$ tends to $\infty$.
%\end{enumerate}
%We will denote by $U^{\mu}(\Omega)$ the set of universal series.
%\end{definition}

%Notice the difference of notations between the set $\mathcal{U}^{\MS}(\D)$ of Pad\'e universal series (Theorem \ref{DFN-main-intro}) and the set $U^{\mu}(\D)$ of classical universal series. Also observe 
%Notice that item b) is trivial if $\Omega = \D$. 
For sake of simplicity, we restrict ourselves to the case of the unit disc, but the results displayed in this section would hold true for any simply connected domain $\Omega$ in $\C$. 

%We recall the following theorem about universal series, cf. \cite{Nes}.
%
%\begin{theorem}\label{thm-univ-Taylor-series}
%For any increasing sequence $\mu$ of positive integers, the set $U^{\mu}(\D)$ is a dense 
%$\text{G}_{\delta}$-subset of $H(\D)$.
%\end{theorem}

%When $\mu =\N$ we simply denote $U^{\mu}(\D)=U(\D)$.
It may be natural to ask whether a classical universal series provides a Pad\'e universal series, possibly in a systematic way. The answer to this question is affirmative and, even more, one can exhibit a \emph{dense} class of classical universal series which are systematically Pad\'e universal series. Here \emph{dense} refers to the topology of $H(\D)$. Subsequently, on can recover the genericity of Pad\'e universal series, stated in the introduction, as a consequence of the theory of classical universal series. In this connection, let us remark that, in the proof of Theorem \ref{main-thm},
%or \ref{main-thm-Q}, to build the Pad\'e universal series we first build a \emph{particular} power series which is almost a classical universal series in the sense that, for every \emph{convenient} compact $K\subset \C$, every function in $A(K)$ is the uniform limit on $K$ of a subsequence of the partial sums of its Taylor expansion at $0$. Indeed with the notations of the proof of Theorem \ref{main-thm}, 
the function $\tilde{f}$ defined in (\ref{def-tildef}), from which the Pad\'e universal series is built, is very close to be a universal series. This observation together with some additional tools from the theory of universal series actually provides us with a way to build up Pad\'e universal series.

%\medskip{}
We first recall the notion of Ostrowski-gaps.
\begin{definition}{\rm Let $\sum_{n\geq 0} a_nz^n$ be a power series. We say that it has Ostrowski-gaps $\left(p_m,q_m\right)_m$ if $(p_m)_m$ and $(q_m)_m$ are sequences of integers such that 
\begin{enumerate}[a)]
\item $p_0<q_0\leq p_1<q_1\leq\ldots\leq p_m<q_m\leq\ldots$ and $\displaystyle{\lim_{m\rightarrow \infty}q_m/p_m=\infty}$,
\item for $\displaystyle{I=\cup_{m\geq 0}\{p_m+1,\dots,q_m\}}$, we have $\displaystyle{\lim_{n\infty,\,n\in I}\vert a_n\vert^{1/n}=0}$.
\end{enumerate}}
\end{definition}
Gehlen, M\"uller and Luh proved that every universal series possesses Ostrowski-gaps \cite{GLM}.
%\begin{theorem}\label{thm-gaps}Let $f=\sum_{n\geq 0}a_nz^n$ in $U\left(\D \right)$. There exist two sequences $(p_m)_m$ and $(q_m)_m$ of integers such for any compact set $K\subset \D^c$ with connected complement and any $h\in A(K)$, there exists a sequence $\left(\lambda _n\right) _n$ of integers such that
%\begin{enumerate}\item $f$ has Ostrowski-gaps $\left(p_m,q_m\right)$ and
%\item $\left\Vert S_{p_{\lambda _n}}\left(f\right)(z)-h(z)\right\Vert _K \rightarrow 0}$, as $n$ goes to $\infty$.
%\end{enumerate}
%\end{theorem}
Recently the authors of \cite{CM} introduced the notion of \emph{large} Ostrowski-gaps where large refers to how fast $q_m/p_m$ tends to $\infty$. More precisely, let us call a \emph{weight}, a strictly increasing function 
$$\varphi:\R_+ \rightarrow \R_+\text{ such that } \varphi(x) \rightarrow \infty\text{ and }\varphi(x)/x \rightarrow 0\text{ as }x\to \infty.$$
A function $f$ in $H(\D)$ is said to have Ostrowski-gaps 
$(p_m,q_m)$ \emph{with respect to the weight $\varphi$} if its Taylor series has Ostrowski-gaps $(p_m,q_m)$ with $p_m<\varphi(q_m)<q_m$ for every $m\geq 0$.
Note that, the faster the ratio $\varphi(x)/x$ decreases to zero, the larger the Ostrowski-gaps are.

%The authors of \cite{CM} proved the existence of a dense $G_{\delta}$ subset of universal series with \emph{large} Ostrowski-gaps. Such universal series are defined as follows.

\begin{definition}\label{defi-large}Let $\varphi$ be a weight and let $\mu$ be an increasing sequence of positive integers. A power series $f$ in $H(\D)$ is a \emph{universal series having Ostrowski-gaps with respect to $\varphi$} if, for every compact set $K\subset\C \setminus \D$ with connected complement and every function $h$ in $A(K)$, there exist a subsequence $(p_m)$ of $\mu$ and a sequence $(q_m)$ such that
\begin{enumerate}[a)]
\item $S_{p_m}(f)\rightarrow h$ in $A(K)$ as $m\rightarrow \infty$,
\item $f$ has Ostrowski-gaps $(p_m,q_m)_m$ with respect to $\varphi$.
\end{enumerate}
\end{definition}
We denote by $U^{(\mu,\varphi)}(\D)$ the set of such power series. It has been proved in \cite{CM} that, for any weight $\varphi$, the set $U^{(\mu,\varphi)}(\D)$ is residual in $H(\D)$.

%\begin{remark}Notice that the sequence $\left(p_m\right)_n$ plays a particular.
%\end{remark}

\begin{proposition}[{\cite[Proposition 2.5]{CM}}]\label{prop-large}
Let $\varphi$ be a weight and let $\mu$ be an increasing sequence of positive integers. The set $U^{(\mu,\varphi)}(\D)$ is a dense $G_{\delta}$ subset of $H(\D)$.
\end{proposition}

Actually it can be checked that the proof of Proposition \ref{prop-large} easily implies the following one stating the density of a very particular class of universal series.

\begin{proposition}\label{prop-large-zero}Let $\varphi$ be a weight and let $\mu$ be an increasing sequence of positive integers. 
There exists a function $f=\sum_{k\geq 0}a_{k}z^{k}$ in $H(\D)$, having Ostrowski-gaps $\left(p_m,q_m\right)_m$ with respect to $\varphi$, such that $\left(p_m\right)_m$ is a subsequence of $\mu$, $a_{p_m}\neq 0$ for every $m$, and for every compact set $K\subset\C \setminus \D$ with connected complement and every function $h$ in $A(K)$, there is a sequence $\left(\lambda _n \right) _n$ of integers with the following properties:
\begin{enumerate}[a)]
\item $\left\Vert S_{p_{\lambda _n}}(f)-h\right\Vert _K\rightarrow 0,\hbox{ as }n\rightarrow \infty$,
\item $a_k=0$, for $p_m\leq k \leq q_m$, $m\geq 0$.
\end{enumerate}
The set $U_0^{(\mu,\varphi)}(\D)$ of such functions is a dense subset of $H(\D)$.
\end{proposition}

\begin{remark}Notice that, in the above proposition, the gaps $\left(p_m,q_m\right)_m$ do not depend on the compact set $K$ nor on the function $h\in A(K)$, compare with Definition \ref{defi-large}.
\end{remark}

We first prove that any function of $U_0^{(\mu,\varphi)}(\D)$ provides a Pad\'e universal series whose Pad\'e approximants have prescribed poles.

\begin{proposition}\label{prop-csq-q}Let $\varphi$ be a weight and $\mu$ an increasing sequence of positive integers. Let $g=\sum _{k\geq 0}b_kz^k$ be in $ U_0^{(\mu,\varphi)}(\D)$. We denote by $Z_g$ the union of all zeros of the partial sums $S_{p_m}(g)$, $m\geq 0$, where $\left(p_m,q_m\right)_m$ are the Ostrowski-gaps of $g$.
% given by Proposition \ref{prop-large-zero}. 
Let now $Q$ be any polynomial of degree $q \geq 1$, with zeros $Z(Q)$ outside $\D\cup Z_g$ and such that $Q(0)=1$. Then the function $f:=g/Q$ belongs to the set $\UU_{Q}^{\mu}(\D)$, as defined in Theorem \ref{main-thm-Q}.
%$ \in H(\D)$ satisfies the following property: for every compact set $K\subset\C \setminus (\D \cup Z(Q))$ with connected complement and every function $h$ in $A(K)$ there exists a sequence $\left(\lambda _n \right) _n$ of integers such that
%\begin{enumerate}[a)]
%\item For every $n\geq 0$, $f\in \DD {p_{\lambda _n},q}$;
%\item $[f;p_{\lambda _n}/q]\rightarrow h$ in $A(K)$ as $n\rightarrow \infty,$
%\item $[f;p_{\lambda _n}/q]\rightarrow f$ in $H(\D)$ as $n\rightarrow \infty$.
%\end{enumerate}
\end{proposition}

\begin{proof}Let $K$ be any compact set in $\C \setminus (\D\cup Z(Q))$ and let $h \in A(K)$. %$g$ belongs to $U_0^{(\mu,\varphi)}(\D)$ so t
There exists a sequence $\left(\lambda _n\right)_n$ of integers such that
\begin{enumerate}
\item $S_{p_{\lambda _n}}(g)\to Qh$ in $A(K)$, as $n$ tends to $\infty$;
\item $S_{p_{\lambda _n}}(g)\to g$ in $H(\D)$, as $n$ tends to $\infty$;
\item For every $k\in \left\{p_m,\ldots,q_m\right\}$, $m\geq 0$, we have $b_k=0$.
\end{enumerate}
Now, by the properties of Ostrowski-gaps, up to a subsequence, we can assume that $q_{\lambda _n}-p_{\lambda _n}> q$ for any $n\geq 0$. Hence, $f\in\DD_{p_{\lambda_{n}},q}$ and $[f;p_{\lambda _n}/q]=S_{p_{\lambda _n}}(g)/Q$, where we use the fact that $Q$ and $S_{p_m}(g)$ are coprime. Dividing by $Q$ in items 1. and 2. finishes the proof.
\end{proof}

\begin{remark}Observe that the previous proof works for any weight $\varphi$.
\end{remark}

%We immediately deduce the following particular case of Theorem \ref{DFN-main}.
%\begin{corollary}\label{coro-univ-pade-q}Let $\MS:=\left(p_n,q\right)_n$ be a sequence of elements in $\N \times \N$ such that $\left(p_n\right)_n$ is a strictly increasing sequence and $q\geq 1$ an integer. The set $\mathcal{U}^{\MS}(\D)$ is a dense $G_{\delta}$ subset of $H(\D)$.
%\end{corollary}

%\begin{proof}This is the standard proof by Baire Category Theorem. It suffices to observe that given $g$ in the dense set $\mathcal{U}_0^{(\mu,\varphi)}(\D)$ with $\mu =\left(p_n\right)_n$ and $\varphi$ any weight, $g$ approaching a polynomial $P$, there exists a polynomial $Q$ of degree $q$ as much as we want close to $1$ uniformly on a neighbourhood of $\overlin{\D}$, with zeros outside $Z_g$. Thus $f:=g/Q$ is universal in the sense of Proposition \ref{prop-csq-q} and approaches $P/Q$ in $H(\D)$ which is close to $P$.
%\end{proof}

%\medskip{}

Next, we recall \cite[Theorem 3.1]{DFN} in the case of the unit disk.

\begin{theorem}[Theorem 3.1 of \cite{DFN}]\label{thm-DFN}Let $\mathcal{S}=\left(p_m,q_m\right)_m$ be a sequence of pairs of positive integers such that $\left(p_n\right)_n$ is unbounded. There exists a function $f$ in $H(\D)$ such that for every compact set $K\subset \C\setminus \D$ with $\C\setminus K$ connected and every function $h\in A(K)$, there is a subsequence $\left(\lambda_n\right)_n$ of $\N$ with the following properties:
\begin{enumerate}[a)]
\item $f\in \DD _{p_{\lambda _n},q_{\lambda _n}}$ for all $n=1,2,\ldots$;
\item $[f;p_{\lambda _n}/q_{\lambda _n}]\rightarrow h$ in $A(K)$ as $n\rightarrow \infty$;
\item $[f;p_{\lambda _n}/q_{\lambda _n}]\rightarrow f$ in $H(\D)$ as $n\rightarrow \infty$;
\end{enumerate}
The set $\UU ^{\mathcal{S}}(\D)$ of those functions is a dense $G_{\delta}$ subset of $H(\D)$.
\end{theorem}

The following proposition tells us that each element of certain classes of universal series with large Ostrowski-gaps is a Pad\'e universal series in the sense of \cite[Theorem 3.1]{DFN}.

\begin{proposition}\label{last-prop-end}Let $\mathcal{S}=\left(p_m,q_m\right)_m$ be a sequence of pairs of positive integers such that $\left(p_n\right)_n$ is unbounded. There exist a weight $\varphi$ and an increasing sequence $\mu$ of positive integers such that $U_0^{(\mu,\varphi)}(\D)\subset \UU ^{\mathcal{S}}(\D)$.
\end{proposition}

\begin{proof}Let $\mu =\left(p_m\right)_m$ and $\varphi$ be such that, for any positive integer $n$, 
$$\varphi(p_r+q_r)<n,\quad\text{where}\quad p_r=\min\{p_m;\,p_m\geq n\}.$$
Let $g=\sum _k a_kz^k \in U_0^{(\mu,\varphi)}(\D)$. By the choice of $\mu$ and $\varphi$, it follows that there exists a sequence $\left(\lambda _n\right)_n \subset \N$ such that
\begin{equation}\label{gap-gd}
\tilde{p_n}=p_{\lambda _n}\text{ and }\tilde{q}_n \geq p_{\lambda_n}+q_{\lambda _n}.
\end{equation}
To prove that $g\in \UU ^{\mathcal{S}}(\D)$, we fix a compact set $K\subset \C \setminus \D$ with $\C \setminus K$ connected and $h\in A(K)$. Since $g\in U_0^{(\mu,\varphi)}(\D)$ we have, up to a subsequence,
\begin{equation}\label{eq2}
\|S_{\tilde{p}_n}(g)-\tilde{h}\|_K\rightarrow 0\mbox{ and } 
\left\Vert S_{\tilde{p}_n}(g)-g\right\Vert _L\rightarrow 0\mbox{ as }n\rightarrow \infty.
\end{equation}
By (\ref{gap-gd}) $S_{\tilde{p}_n}(g)=S_{p_{\lambda _n}}(g)$ and $a_k=0$ for any $k\in [p_{\lambda _n},p_{\lambda _n}+q_{\lambda _n}]$, so $S_{p_{\lambda _n}}(g)=[g;p_{\lambda _n}/q_{\lambda _n}]$, hence the conclusion by (\ref{eq2}).
\end{proof}

We are in a position to recover \cite[Theorem 3.1]{DFN} for the case of a disk, that is that $\UU^{\MS}(\D)$ is a dense $G_{\delta}$ subset of $H(\D)$, without using Baire Category Theorem:

Let $(L_r)_r$ be an exhaustion of compacta in $\D$, $\left(K_{m}\right)_m$ a sequence of compacta in $\C \setminus \D$ given by \cite[Lemma 5]{Nes} and $\left(P_j\right)_j$ an enumeration of polynomials with coefficients in $\Q+i\Q$. Then
$$\UU ^{\MS}(\D)=\bigcap_{j,s,r,m}\,\bigcup_{n\in \N}F(n,j,s,r,m)$$
where $F(n,j,s,r,m)$ is the set of function $f$ in $H(\D)$ such that
$$f\in \DD_{p_n,q_n},\quad\Vert [f;p_n/q_n]-P_j\Vert_{K_{m}}<s^{-1},
\quad\Vert [f;p_n/q_n]-f\Vert_{L_r}<s^{-1}.
$$
We know that each set $F(n,j,s,r,m)$ is open so $\UU ^{\MS}(\D)$ is a $G_{\delta}$ subset of $H(\D)$. By Proposition \ref{last-prop-end} it contains some $U_0^{(\mu,\varphi)}(\D)$ which is dense by Proposition \ref{prop-large-zero}, hence the conclusion.

We end up with a statement which is slightly stronger than \cite[Theorem 3.1]{DFN} (for the unit disk).

\begin{corollary}With the notations and assumptions of Theorem \ref{thm-DFN}, the subset of $\UU ^{\mathcal{S}}(\D)$ each element of which belongs to $\NN _{m,n}$ for any $(m,n)\in \N \times \N$, is a dense $G_{\delta}$ subset of $H(\D)$.
\end{corollary}

\begin{proof}This is nothing but Baire Category Theorem together with \cite[Theorem 3.1]{DFN}, since $\NN_{m,n}$ is an open dense subset of $H(\D)$.
\end{proof}

%\begin{remark} The above results could also be stated with formal power series instead of functions in $H(\D)$.\\
%(2) Results of Section 3 and 4 are of different kinds. In the present section we fix a universal series as well as the denominators of the Pad\'e approximants and then we show that they actually have universal properties (very close to usual Pad\'e universality). In Section 3 we impose the \emph{general form} of the Pad\'e universal series and the universal properties to get, and afterward we construct a universal series and the Pad\'e approximants which will have these universal properties.
%\end{remark}


\begin{thebibliography}{99}

\bibitem{Aron1} \textsc{R. Aron, D. Garc\'ia, M. Maestre}, {\it Linearity in non-linear problems}, Rev. R. Acad. Cien. Ser. A Mat. \textbf{95} (2001) 7--12.

\bibitem{Aron2} \textsc{R. Aron, V. I. Gurariy, J. B. Seoane-Sep\'ulveda}, {\it Lineability and spaceability of sets of functions on $\R$}, Proc. Amer. Math. Soc. \textbf{133} (2005) 795--803.

\bibitem{BGM} \textsc{G. A. Baker Jr., P. R. Graves-Morris}, Pad\'e Approximants I and II, Cambridge Univ. Press, 1996.

\bibitem{Bar} \textsc{L. Baratchart}, {\it Existence and generic properties for $L^{2}$ approximants of linear systems}, IMA J. of Math. Control Inform., \textbf{3} (1986),
89--101.

\bibitem{Bay0} \textsc{F. Bayart}, {\it Topological and algebraic genericity of divergence and universality}, Studia Math. \textbf{167} (2005) 161--181.

\bibitem{Bay} \textsc{F. Bayart}, {\it Linearity of sets of strange functions,} Michigan Math. J. \textbf{53} (2005) 291--303.

\bibitem{bgnp} \textsc{F. Bayart, K.-G. Grosse-Erdmann, V. Nestoridis, C. Papadimitropoulos}, {\it Abstract theory of universal series and applications,} Proc. London Math. Soc. \textbf{96} (2008) 417--463.

\bibitem{BO} \textsc{L. Bernal-Gonz\'alez, M. Ord\'o\~nez Cabrera}, {\it Lineability criteria, with applications,} J. Funct. Anal. \textbf{266} (2014), 3997--4025.

\bibitem{BPS} \textsc{L. Bernal-Gonz\'alez, D. Pellegrino, J. B. Seoane-Sep\'ulveda}, {\it Linear subsets of non-linear sets in topological vector spaces,} Bull. Amer. Math. Soc., to appear.

\bibitem{BOR} \textsc{P.B. Borwein}, {\it The usual behaviour of rational approximants,} Canad. Math. Bull. \textbf{26} (1983) 317--323.

\bibitem{BOR2} \textsc{P.B. Borwein, S.P. Zhou}, {\it 
The usual behavior of rational approximation II,}
J. Approx. Theory \textbf{72} (1993), 278--289. 

\bibitem{CH}\textsc{S. Charpentier}, \textit{On the closed subspaces of universal series in Banach and Fr\'echet spaces}, Studia Math. \textbf{198} (2010) 121--145.

\bibitem{CM} \textsc{S. Charpentier, A. Mouze}, {\it Universal Taylor series and summability}, preprint.

\bibitem{DFN} \textsc{N. Daras, G. Fournodavlos, V. Nestoridis}, {\it Universal Pad\'e approximants on simply connected domains}, preprint.

\bibitem{FNes} \textsc{G. Fournodavlos, V. Nestoridis}, {\it Generic approximation of functions by their Pad\'e approximants,} J. Math. Anal. Appl. \textbf{408} (2013), 744--750.

\bibitem{GLM} \textsc{W. Gehlen, W. Luh, J. M\"uller}, {\it On the existence of O-universal functions,} Complex Variables Theory Appl. \textbf{41} (2000), 81--90.

%\bibitem{Gon1}\textsc{A. A. Gonchar}, \textit{Properties of functions related to their rate of approximability by rational functions}, Amer. Math. Soc. Transl. \textbf{91} (1970), 99--128.

\bibitem{Gon} \textsc{A.A. Gonchar}, {\it Uniform convergence of diagonal Pad\'e approximants}, Math. USSR-Sb. 46 (1983), 539--559.

\bibitem{GRA} \textsc{W.B. Gragg}, {\it The Pad\'e table and its relation to certain algorithms of numerical analysis}, SIAM Review {\bf 14} (1972), 1--61.

\bibitem{G-E} \textsc{K.-G. Grosse-Erdmann, A. Peris}, {\it Linear Chaos}, Springer, New York, 2011.

\bibitem{Her} \textsc{G. Herzog}, {\it On universal functions and interpolation}, Analysis \textbf{11} (1991), 21--26. 

\bibitem{Her2} \textsc{G. Herzog}, {\it Lagrange interpolation for the disk algebra: the worst case.}, J. Approx. Theory \textbf{115} (2002), 354--358.

\bibitem{Luh2}\textsc{W. Luh}, \textit{Universal approximation properties of overconvergent power series on open sets}, Analysis \textbf{6} (1986), 191--207.

%\bibitem{MNadv} \textsc{A. Melas, V. Nestoridis}, {\it Universality of Taylor Series as a Generic Property of Holomorphic Functions,} Adv. Math., \textbf{157} (2001), 138--176.

\bibitem{Men}\textsc{Q. Menet}, \textit{Sous-espaces ferm\'es de s\'eries universelles sur un espace de Fr\'echet}, Studia Math. \textbf{207} (2011) 181--195.

\bibitem{Men2}\textsc{Q. Menet}, \textit{Hypercyclic subspaces and weighted shifts}, Adv. Math. \textbf{255} (2014) 305--337.

\bibitem{Montes} \textsc{A. Montes-Rodr\'iguez}, {\it Banach spaces of hypercyclic vectors}, Michigan Math. J. {\bf 43} (1996), 419--436.

\bibitem{Nes} \textsc{V. Nestoridis}, {\it Universal Taylor series,} Ann. Inst. Fourier \textbf{46} (1996), 1293--1306.

\bibitem{Nes2} \textsc{V. Nestoridis}, {\it An extension of the notion of universal Taylor series,} Computational methods and function theory 1997 (Nicosia), 421--430, Ser. Approx. Decompos., \textbf{11}, World Sci. Publ., River Edge, NJ, 1999.

\bibitem{Nes3} \textsc{V. Nestoridis}, {\it Universal Pad\'e approximants with respect to the chordal metric,} Izv. Nats. Akad. Nauk. Armenii Mat. \textbf{47} (2012), 13--34; translation in J. Contemp. Math. Anal. \textbf{47} (2012), no. 4, 168--181.

\bibitem{NS} \textsc{E.M. Nikishin, V.N. Sorokin}, Rational Approximations and Orthogonality,{\em Transl. Amer. Math. Soc.}, Vol. {\bf 92}, Providence, R.I., 1991.
  
\bibitem{Pet} \textsc{H. Petersson}, {\it Hypercyclic subspaces for Fr\'echet space operators,} J. Math. Anal. Appl. \textbf{319} (2006), 764--782.

\bibitem{Per} \textsc{O. Perron}, "Die Lehre von den Kettenbr\"uchen," 3rd ed., Vol. 2, Teubner, Stuttgart, 1957.

\bibitem{Wal} \textsc{H. Wallin}, {\it The convergence of Pad\'e approximants and the size of the power series coefficients}. Applicable Anal., 4 (1974), 235--251.
\end{thebibliography}
\end{document}